\documentclass[12pt,reqno]{amsart}
\usepackage{cases}
\usepackage{amscd}
\usepackage{amsfonts}
\usepackage{amssymb}
\usepackage{amsmath}
\usepackage{graphicx}
\usepackage{epstopdf}
\usepackage{subfigure}
\usepackage{amsthm}
\usepackage{stmaryrd}
\usepackage{geometry}

\theoremstyle{remark}
\newtheorem{example}{\textbf{Example}}[section]
\numberwithin{equation}{section}

\usepackage{color}
\usepackage{datetime}

\def\tr{\textcolor{red}}
\def\tb{\textcolor{blue}}

\makeatletter
\newcommand\figcaption{\def\@captype{figure}\caption}
\newcommand\tabcaption{\def\@captype{table}\caption}
\makeatother

\usepackage{geometry}
\usepackage{algorithm}
\usepackage{multirow}
\def\bq{\begin{equation}}
\def\eq{\end{equation}}
\def\bqs{\begin{equation*}}
\def\eqs{\end{equation*}}
\def\ba{\begin{aligned}}
\def\ea{\end{aligned}}
\def\br{\begin{eqnarray}}
\def\er{\end{eqnarray}}
\def\brr{\bq\begin{array}{rlll}}
\def\err{\end{array}\eq}

\def\text#1{\hbox{#1}}
\newtheorem{thm}{Theorem}[section]
\newtheorem{lem}{Lemma}[section]

\newtheorem{rem}{Remark}[section]

\newcommand{\bsub}{\begin{subequations}}
\newcommand{\esub}{\end{subequations}$\!$}
\title[]{A mixed discontinuous Galerkin method without interior penalty for time-dependent fourth order problems}
\author[H.~Liu, P. ~Yin]{Hailiang Liu and Peimeng Yin}
\email{hliu@iastate.edu; pemyin@iastate.edu}
\address{ Iowa State University, Mathematics Department, Ames, IA 50011}
\keywords{Biharmonic equation, discontinuous Galerkin method, $L^2$ error estimates, Swift-Hohenberg equation}
\subjclass{65N15, 65N30, 35J40}
\date{December  28, 2017}
\begin{document}

\begin{abstract}
A novel discontinuous Galerkin (DG) method is developed to solve time-dependent bi-harmonic type equations involving fourth derivatives in one and multiple space dimensions.  We present the spatial DG discretization based on a mixed formulation and central interface numerical fluxes so that the resulting semi-discrete schemes are $L^2$ stable even without interior penalty.  %We first rewrite the equation into a second order system, and then spatially discrete the weak formulation of the system by DG scheme without penalty terms.
For time discretization, we use Crank-Nicolson so that the resulting scheme is unconditionally stable and second order in time. % The stability is naturally obtained due to the cancellation of the bilinear functional in the system when taking summation.
We present the optimal $L^2$ error estimate of $O(h^{k+1})$ for polynomials of degree $k$ for semi-discrete DG schemes, and the $L^2$ error of $O(h^{k+1} +(\Delta t)^2)$ for fully discrete DG schemes.
 % by introducing a global projection in one-dimension or its tensor product in multi-dimension.
Extensions to more general fourth order partial differential equations and cases with non-homogeneous boundary conditions are provided. Numerical results are presented to verify the stability and accuracy of the schemes. Finally, an application to the one-dimensional Swift-Hohenberg equation endowed with a decay free energy is presented.
%and further explore the potential of the DG method to nonlinear problems.
\end{abstract}

\maketitle

\bigskip

\section{Introduction}
In this paper, we are interested in discontinuous Galerkin approximations to the fourth order partial differential equations (PDEs) of the form
\begin{subequations}\label{Beam}
\begin{align}
u_t & = Lu \quad  x\in \Omega \subset \mathbb{R}^d, \quad t>0,\\
u(x, 0) & =u_0(x), \; x\in \Omega,
\end{align}
\end{subequations}
where $L=\sum_{m=0}^2a_m \Delta^m$ is a linear differential operator of fourth order and $a_m (m=0,1,2)$ are constants with $a_2<0$, $\Omega$ is a bounded rectangular domain in $\mathbb{R}^d$, $u_0(x)$ is a given function.  Our analysis is presented mostly for periodic boundary conditions,  extensions to other non-homogeneous boundary conditions will then follow. The model could include a lower order term such as $f(u,x, t)$, without additional difficulty.
%dealt with using proper boundary fluxes.

The fourth order PDEs appear often in physical and engineering applications, such as the modeling of the thin beams and plates, strain gradient elasticity, %\cite{EGHLMT02},
thermal convection, and phase separation in binary mixtures.
%the theory of phase transitions
 %\cite{DS88}.
 The special cases of (\ref{Beam}) include the linear time-dependent biharmonic equation with
\bqs%\label{BH}
L=-\Delta^2,
\eqs
and the linearized Cahn-Hilliard equation
\bqs%\label{CH}
L=-\Delta^2 -\Delta.
\eqs
%Generally, we can separate the numerical method of solving the time-dependent or steady state of the fourth-order equation (\ref{Beam}) into three categories.  The first type is to design the numerical method directly from the original equation (\ref{Beam}), see \cite{F13, CS08}; the second type is based on the so-called mixed formulation that is to rewrite (\ref{Beam}) into two coupled second-order equations, such as \cite{L06}; the last type is to rewrite (\ref{Beam}) into a system of first-order equations, and this is mainly for LDG method and related work can be found in \cite{YS02, DS09, MSW12}.
In the literature, various numerical methods have been developed to discretize fourth order partial differential equations, such as  mixed finite element methods (see e.g. \cite{CR74, BR78, Fa78, GP79, BOP80, Mo87}), and finite difference methods (see e.g. \cite{GKO95}). In this paper we will discuss discontinuous Galerkin methods, using a discontinuous Galerkin finite element approximation in the spatial variables coupled with a proper time discretization. It is well known that for equations containing higher order spatial derivatives, discontinuous Galerkin discretization cannot be directly applied. This is because the solution space, which consists of piecewise polynomials discontinuous at the element interfaces, is not regular enough to handle higher derivatives. This is a typical non-conforming case in finite elements.

%\cite{BR78,  S78, FO80, BOP80, Mo87, L06},

One approach to resolve such difficulty is the local discontinuous Galerkin (LDG)  method (see e.g.,  \cite{YS02, DS09, MSW12,  WZS17} for fourth order problems). The idea is to suitably rewrite the higher order equation into a first order system and then discretize it by the DG method \cite{CS98}.
%the discontinuous Galerkin method on the system.
 The local  numerical fluxes without interior penalty can be designed to guarantee stability. %  and local solvability of all the auxiliary variables introduced to approximate the derivatives of the solution to the steady problem.
The LDG method has been successful in handling equations with high-order derivatives, since it was first developed by Cockburn and Shu \cite{CS98} for the second order convection diffusion equation. However, these schemes increase the number of unknowns in numerical solutions.
%One obvious drawback of the LDG method when applied to the fourth order problems is a widened stencil, which results in additional complexity and computational cost in solving all auxiliary variables.

Another approach is to weakly impose the inter-element continuity conditions using interior penalties.  In the context of finite element framework,  $C^1$ conforming finite element methods for the biharmonic equation is known computationally intensive due to the imposition of $C^1$-continuity across the  element interfaces, several non-conforming approaches such as  $C^0$-interior penalty methods \cite{BS05, EGHLMT02}  and interior penalty methods \cite{Ba77, MS03, MSB07, SM07, GV15} have been proposed. These approaches use either continuous or discontinuous finite element solution spaces in which continuity conditions are weakly enforced through interior penalties.  A related strategy is the direct DG discretization based on numerical fluxes  which penalize jumps of derivatives when crossing element interfaces \cite{CS08}. For DG schemes with interior penalties, the practical choice of penalty parameters is often a subtle matter.

In this work we reformulate the fourth order PDEs into a second order coupled system and discretize the system by a DG method without interior penalty. In the case $L=-\Delta^2$, such reformulation
\begin{equation}\label{BeamSys}
\left \{
\begin{array}{rl}
u_t = & \Delta q, \\
q = & -\Delta u,
\end{array}
\right.
\end{equation}
 is the usual mixed formulation \cite{C78}, which
 %has been extensively used  to design mixed finite element methods.
 has been used to design the mixed DG methods with interior penalties  in \cite{GNP08, XBLM17} for solving the biharmonic equation.  Our DG method derives from a direct DG discretization of the mixed formulation (\ref{BeamSys}). Instead of the standard DG ansatz analogous to the discretization of diffusion,  the simplest form for numerical fluxes is used:  the arithmetic mean of the solution gradient and the arithmetic mean of the solution. The resulting scheme is the most simple variant to date for the discretization of second order terms,  i.e.,  without any interior penalty. This is in sharp contrast to the DDG methods introduced in \cite{LY09, LY10} for diffusion, where interface corrections are included to penalize jumps of both the numerical solution  and its second order derivatives. With formulation (\ref{BeamSys}),  stability of the resulting DG scheme is naturally ensured due to the symmetric nature of the underlying bilinear operator.
 %without reference to any  interior penalty. %In addition, the scheme is generic in the sense that it is independent of the underlying form of the derivative terms, as only arithmetic means are used at the interfaces.
 It is also parameter free, i.e. no particular choice of any penalty constant is necessary. This makes the scheme simple to implement for generic linear and non-linear problems.

%The theory of a priori error estimates for various mixed finite element methods has been well-developed \cite{BR78,  S78, FO80, BOP80, Mo87, L06}, using certain interpolation or specially paired solution-test function spaces.
It is known that for DG methods stability itself does not necessarily imply the optimal convergence. Obtaining optimal error estimates for DG methods has been a major subject of research. The a priori error estimate results for DG methods with interior penalties have been reported in \cite{MS03, MSB07, SM07, GNP08, CS08, XBLM17} for biharmonic type equations, in these works penalty parameters play a special role in both the stability analysis and the error estimates.
%For steady problems, the corresponding DG schemes produce efficient algorithms with fewer unknowns, generating systems of equations with positive-definite matrices. %this is in contrast to the error estimates in \cite{DS09}  for the LDG method to linear biharmonic type equations.

The main quest in this article is whether optimal convergence can still be achieved without interior penalty.  We carry out the optimal $L^2$ error estimates for both semi-discrete and fully-discrete schemes with periodic boundary conditions, in  both one and multi-dimensions.  The crucial ingredient in the one-dimensional error analysis is a global projection $P$ defined by $A(v-Pv,\phi)=0$ for any test function $\phi$ in the finite element space, and the corresponding projection error. Here $A(\cdot, \cdot)$ is the bilinear operator obtained by the penalty-free DG discretization of the operator $-\partial^2_x $.  In multi-dimensional case, we use the tensor product polynomials of degree at most $k$, and make use of the projection error  obtained in \cite{L15} and the bilinear form estimate $|A(v-Pv,\phi)| \leq Ch^{k+2}|v|_{k+2}\|\phi\|$ obtained in \cite{LHLY17}.  A related work is \cite{BEMS07}, in which the authors
use the inf-sup strategy to prove the optimal $L^2$ convergence rates for the symmetric DG method without interior penalty using $P^k(k \geq 2)$ polynomials for one dimensional second order elliptic problems.

Extension to more general equations of form (\ref{Beam}) is carried out by rewriting $L$ as  $L=- {\mathcal L}^2 +M$, where $M=a_0-\frac{a_1^2}{4a_2}$ and $\mathcal L=\sqrt{-a_2} \left( \Delta +\frac{a_1}{2a_2} \right)$ is a second order  operator, and the optimal $L^2$ error estimate can also be obtained.  For three typical non-homogeneous  boundary conditions we present DG schemes with boundary corrections. Boundary penalty is needed in some cases to weakly enforce the given boundary data, as usually done for the weak formulation of elliptic problems \cite{L68}.  In fact, imposing boundary conditions only weakly is one of the main advantages of the DG methods to  boundary-value problems for higher order PDEs such as (\ref{Beam}a).

% In fact, the use of a penalty formulation for enforcing the boundary conditions can be traced back to the late 1960s, see Lions \cite{L68}, in which the Dirichlet boundary condition was enforced in the weak formulation of elliptic problems.

The rest of the paper is organized as follows: in section 2, we describe the mixed DG methods in one dimension and present the optimal error estimates for both semi-discrete and fully discrete schemes to time-dependent biharmonic problems. In section 3, we formulate the DG scheme in multi-dimensions along with its stability and optimal error estimates using tensor product polynomials.  In section 4, we extend the DG schemes to more general fourth order time-dependent PDEs,  cases with non-periodic boundary conditions, and the one-dimensional Swift-Hohenberg equation -- a nonlinear problem with a decay free energy \cite{SH77}.  Several numerical results are presented in section 5 to verify  the stability and accuracy of the schemes. Finally, we give  concluding remarks in section 6 to summarize results in this paper and indicating future work.

\section{The DG scheme in one dimension}
In this section we consider the one dimensional time-dependent fourth order equation (\ref{Beam}), i.e.,
\begin{equation}\label{Beam1D}
u_t= - u_{xxxx} \quad x\in [a,b], \;  t>0,
\end{equation}
subject to initial data $u(x, 0)=u_0(x)$, and periodic boundary conditions.

We partition the interval $[a,b]$ into computational cells $I_j=[x_{j-1/2},x_{j+1/2}]$, with $x_{1/2}=a$ and $x_{N+1/2}=b$, and mesh size $h_j=x_{j+1/2}-x_{j-1/2}$, with $h=\max_{1\leq j \leq N}  h_j$.   And we define the finite element space
$$
V_h^k=\{ v \in L^2([a, b])\; : \; v |_{I_j} \in P^k(I_j), \; j=1,2, \cdots, N \},
$$
where $P^k(I_j)$ denotes the set of all polynomials of degree at most $k$ on $I_j$.   At cell interfaces  $x=x_{j+1/2}$ we use the notation
$$
v^\pm=\lim_{\epsilon \to 0}v(x \pm \epsilon), \quad \{v\}=\frac{v^-+v^+}{2}, \quad [v]=v^+-v^-.
$$
Based on its mixed  formulation,
\begin{equation}\label{BeamSys1D}
%\left \{
%\begin{array}{rl}
    u_t =  q_{xx}, \;
    q = -u_{xx},
%\end{array}
%\right.
\end{equation}
the DG scheme for (\ref{Beam1D}) is to find $(u_h, q_h) \in V_h^k \times V_h^k$ such that for all $\phi, \ \psi \in V_h^k$
and  $j =1, 2 \cdots, N$,
\begin{subequations}\label{FPDGCell}
\begin{align}
    \int_{I_j} u_{ht} \phi dx = &  - \int_{I_j} q_{hx} \phi_x dx+ \widehat{(q_{hx})} \phi |_{\partial I_j} + (q_h - \widehat{q_h})\phi_x|_{\partial I_j}, \\
   %   & + \alpha\left( - \int_{I_j} u_x \phi_x dx+ \widehat{(u_x)} \phi |_{\partial I_j} + (u - \widehat{u})\phi_x|_{\partial I_j} \right) +\int_{I_j} f \phi dx,\\
    \int_{I_j} q_h \psi dx = & \int_{I_j} u_{hx} \psi_x dx - \widehat{(u_{hx})} \psi |_{\partial I_j}
     -(u_h-\widehat{u_h})\psi_x|_{\partial I_j},
\end{align}
\end{subequations}
where the notation $v|_{\partial I_j}=v^-_{j+1/2}-v^+_{j-1/2}$ is used,  and on each cell interface $x_{j+1/2}, j=0, 1, 2, \cdots N$, the numerical fluxes are given by
\begin{equation}\label{DDGFlux}
\begin{aligned}
&    \widehat{q_{hx}} =  \{q_{hx} \},\quad  \widehat{q_h} = \{q_h \},\\
&    \widehat{u_{hx}} =  \{u_{hx} \},  \quad   \widehat{u_h} = \{u_h \},
\end{aligned}
\end{equation}
where $\{v\}_{1/2}=\{v\}_{N+1/2}$ is understood as $\frac{1}{2}(v_{1/2}^+ + v_{N+1/2}^-)$ for $v=u_h, q_h, u_{hx}$ and $q_{hx}$.
%Note that at $x=\{a, b\}$ periodic boundary conditions are called in use.
The initial data for $u_h$ is taken as the piecewise $L^2$ projection of $u_0(x)$, that is, $u_h(x, 0)\in V_h^k$ such that
\bq\label{LocL2Proj}
\int_{I_j} (u_0(x)-u_h(x,0))\phi(x)dx=0, \quad \forall \phi\in P^k(I_j),  \quad j=1, \cdots, N.
\eq
Note that  $q_h(x,0) \in V_h^k$ can be obtained from $u_h(x, 0)$ by solving (\ref{FPDGCell}b).

\subsection{Stability and $L^2$ error estimate } We proceed to verify the $L^2$ stability of the above semi-discrete DG scheme and further obtain the optimal $L^2$ error estimate. To this end, we sum
(\ref{FPDGCell}) over $ j=1, \cdots, N$ to obtain
\begin{subequations}\label{FPDG1D}
\begin{align}
(u_{ht}, \phi) = &- A(q_h,\phi),\\
(q_h, \psi) = & A(u_h,\psi),
\end{align}
\end{subequations}
where $(\cdot, \cdot)$ denotes the inner product of two functions over $[a, b]$, and the bilinear functional
\bq\label{bilinear}
A(w,v)=\sum_{j=1}^N \int_{I_j} w_x v_x dx + \sum_{j=1}^N \left(\{w_x\} [v] + [w]\{v_x\} \right)_{j+1/2},
\eq
where by $(\cdot)_{j+1/2}$ we mean evaluation of involved quantities at $x_{j+1/2}$.  Note that $A(\cdot, \cdot)$  is symmetric, that is,
\begin{equation}\label{sym}
A(w,v)=A(v,w).
\end{equation}
For scheme (\ref{FPDG1D}) with (\ref{bilinear}) the following stability result holds.
\begin{thm}\label{SemiStab} ($L^2$-Stability). The numerical solution $u_h$  satisfies
\bq \label{2.8}
\frac{1}{2}\frac{d}{dt} \int_a^b u_h^2dx=-\int_a^b q_h^2dx \leq 0.
\eq
\end{thm}
\begin{proof}
Taking $\phi = u_h$  in  (\ref{FPDG1D}a),  and $\psi = q_h$ in (\ref{FPDG1D}b) respectively, we obtain
%\begin{subequations*}\label{StabSub}
\begin{align*}
\frac{1}{2} \frac{d}{dt} \int_a^b u_h^2  dx  =  -A(q_h, u_h),\;
\|q_h\|^2 =A(u_h, q_h),
\end{align*}
%\end{subequations}
\iffalse
Plugging (\ref{StabSub}b) and (\ref{StabSub}c) into (\ref{StabSub}a), it follows that
\bq
\ba
\frac{1}{2} \frac{d}{dt} \int_I u^2  dx = & -\|q\|^2 - \alpha(q,u) + (f, u) \\
\leq & -\|q\|^2 + | \alpha | \|q\|\|u\| + \|f\|\|u\| \\
\leq & \frac{\alpha^2}{4} \|u\|^2 + \|f\|\|u\|,
\ea
\eq
which gives
\bq\label{StabDE}
\frac{d}{dt}\|u\| \leq \frac{\alpha^2}{4} \|u\| + \|f\|.
\eq
Solving (\ref{StabDE}), then it follows
\bq
\|u\| \leq e^{\frac{\alpha^2}{4}t} \left( \|U_0\| + \int_0^t e^{-\frac{\alpha^2}{4}t}\|f\| dt  \right),
\eq
\fi
which when using (\ref{sym}) implies (\ref{2.8}).
\end{proof}

In order to estimate the $L^2$ error, we introduce a global projection: % which will be used in the error estimate.
for a given piecewise smooth function  $w \in L^2([a,b]), w|_{I_j} \in H^{s+1}(I_j), s\geq k \geq 1$, we define $P w \in V_h^k$ by
\begin{subequations}\label{pbproj}
\begin{align}
    & \int_{I_j} \left( P w (x)-w(x) \right)v(x)dx=0, \quad \forall v \in P^{k-2}(I_j),\\
    & \{ (P w)_x\}_{j+1/2}= \{w_x\}_{j+1/2},\\
    & \{P w\}_{j+1/2} := \{w\}_{j+1/2},
\end{align}
\end{subequations}
for $j=1, \cdots, N$, where $\{v\}_{N+1/2}$ is understood as $\frac{1}{2}(v_{1/2}^+ + v_{N+1/2}^-)$. Note that for $k=1$, (\ref{pbproj}a)
is redundant.

\begin{lem}\label{projexist} For $k = 1$ with  $N$ odd, or any $k\geq 2$, there exists a unique projection $P$ defined by (\ref{pbproj}). Moreover,
\begin{equation}\label{AA}
A(Pw-w, v)=0\quad \forall v\in V_h^k.
\end{equation}
\end{lem}
\begin{proof}
(i)  From the more general result  in \cite[Lemma 2.1]{LHLY17} it follows that such $P$ is uniquely defined.\\
(ii)  Relation (\ref{AA}) can be derived from (\ref{bilinear}) using (\ref{pbproj}) and integration by parts once.

\end{proof}

Before going further we recall the following approximation result for projection $P$.
\begin{lem}\cite{L15}\label{ProjErr} (Projection error). Assume that $w\in H^m$ with $m \geq k+1$. Then we have the following projection error
\bq\label{ProjErr1D}
\|w-Pw\| \leq C|w|_{k+1}h^{k+1},
\eq
where $C$ is independent of $h$. Moreover,
$$ %\bq \label{pvv}
Pv=v, \quad \forall v \in V_h^k.
$$ %\eq
\end{lem}
\begin{thm}\label{ThmL2Err} Let $u_h$ be the numerical solution to (\ref{FPDGCell}) with (\ref{DDGFlux}), and $u$ be the smooth solution to problem (\ref{Beam1D}),  then
\bq\label{pbddgenorm}
\|  u_h(\cdot, t)-u(\cdot, t) \| \leq C h^{k+1}, \quad 0\leq t\leq T,
\eq
where $C$ depends on $\sup_{t \in [0,T]}|u_t(\cdot,t)|_{k+1}$, $\sup_{t \in [0,T]}|u(\cdot,t)|_{k+3}$ and  linearly on $T$,  but independent of $h$.
\end{thm}
\begin{proof}
 The consistency of the DG method (\ref{FPDG1D}) ensures that the exact solution $u$ and $q$ of (\ref{BeamSys1D}) also satisfy
\bq\label{FPDGExact}
\begin{aligned}
(u_t, \phi) = & - A(q,\phi),\\
(q, \psi) = & A(u,\psi)\\
\end{aligned}
\eq
for all $\phi \in V_h^k, \psi \in V_h^k$.
Subtracting (\ref{FPDG1D}) from (\ref{FPDGExact}), we obtain the error  system
\bq\label{FPDGErrSys}
\begin{aligned}
((u-u_h)_t, \phi) = & - A(q-q_h,\phi),\\
(q-q_h, \psi) = & A(u-u_h,\psi).
\end{aligned}
\eq
Denote
\bqs
\begin{aligned}
e_1= & Pu-u_h, \quad \epsilon_1= Pu-u, \\
e_2= & Pq-q_h, \quad \epsilon_2= Pq-q, \\
\end{aligned}
\eqs
and take $\phi=e_1, \psi=e_2$ in (\ref{FPDGErrSys}) respectively, we obtain
\begin{subequations}\label{FPDGErrPlug}
\begin{align}
( e_{1t}, e_1) = & (\epsilon_{1t}, e_1) + A(\epsilon_2, e_1) - A(e_2,e_1),\\
(e_2, e_2) = & (\epsilon_2, e_2)- A(\epsilon_1, e_2) + A(e_1,e_2).
\end{align}
\end{subequations}
Summation of (\ref{FPDGErrPlug}a) and  (\ref{FPDGErrPlug}b) gives
\begin{align*}
\frac{1}{2} \frac{d}{dt}\|e_1\|^2 + \|e_2\|^2 = & (\epsilon_{1t}, e_1) + (\epsilon_2, e_2) + A(\epsilon_2, e_1)- A(\epsilon_1, e_2)\\
=&(\epsilon_{1t}, e_1) + (\epsilon_2, e_2),
\end{align*}
where property (\ref{AA}) of projection $P$ has been used.  This yields
\bqs%\label{ErrIneq}
\ba
\frac{1}{2} \frac{d}{dt}\|e_1\|^2 \leq & \| \epsilon_{1t}\| \|e_1\| + \frac{1}{4} \|\epsilon_2\|^2.
%\leq & \frac{d}{dt}\|\epsilon_1\| \|e_1\| + \frac{1}{4} \|\epsilon_2\|^2 \\
%\leq & C_1 h^{k+1} \|e_1\|  + C_1 h^{2(k+1)},\\
\ea
\eqs
By property (\ref{ProjErr1D}), the right hand side is dominated by
%\bq\label{ProjErrApp}
%\begin{aligned}
%\|\partial_t \epsilon_1 \|\leq & C |u_t|_{k+1}h^{k+1},\\
%\|\epsilon_2\| \leq & C |q|_{k+1}h^{k+1} \leq C |u|_{k+3}h^{k+1}. %  = C_0h^{k+1} \sum_{j=1}^N |U|_{k+3, I_j}^2
%\end{aligned}
%\eq
%, where $C_0$ is a constant.\\
%Then, plugging (\ref{ProjErrApp}) into (\ref{ErrIneq}), we have
$
C |u_t|_{k+1}h^{k+1}\|e_1\|+\frac{1}{4}(C |u|_{k+3}h^{k+1})^2.
$
Hence
\bqs%\label{ErrIneq1}
\ba
%\frac{1}{2} \frac{d}{dt}\|e_1\|^2 \leq & \|\partial_t \epsilon_1\| \|e_1\| + \frac{1}{4} \|\epsilon_2\|^2 \\
\frac{1}{2} \frac{d}{dt}\|e_1\|^2 \leq & C_1 h^{k+1} (\|e_1\|  +  h^{k+1}),
\ea
\eqs
where  $C_1=\max\{ C \sup_{t\in[0,T]}|u_t|_{k+1}, \frac{C^2}{4} \sup_{t\in[0,T]}|u|_{k+3}^2\}$. Set $B=\frac{\|e_1\|}{h^{k+1}}$, then
$$
B\frac{dB}{dt} \leq C_1 (B+1),
$$
which upon integration over $[0, t]$ gives
\bq\label{ErrIneqB}
G(B(t)) \leq G(B(0)) + C_1 t, \\
\eq
where
%\bq\label{GX}
$G(s)=s - \ln( s+1)$
%\eq
is an increasing and convex function on $[0, \infty)$. Note that  $B(0)\leq C_2$ for
\bq\label{B0bnd}
\|e_1(\cdot, 0)\| \leq \| Pu_0-u_0\|+\|u_0-u_h(\cdot, 0) \| \leq C_2 h^{k+1}.
\eq
It can be verified that $G^{-1}(s)/s$ is decreasing for $s>0$, note also that $G^{-1}(s)$ is increasing, hence
\bqs% \label{gy}
G^{-1}(s) \leq \frac{G^{-1}(\delta)}{\delta} \max\{s, \delta\}.
\eqs
This with $\delta=G(C_2)$ when inserted into  (\ref{ErrIneqB}) gives
$$%\bq
B(t) \leq G^{-1}\left(C_1T+G(C_2)\right) \leq C_2 +C T,
$$ %\eq
with  $C=\frac{C_1C_2}{G(C_2)}$.  Thus,
$$ %\bq\label{ErrProjDG}
\|e_1(\cdot, t)\| = B(t)h^{k+1} \leq \left(C_2 +C T\right)  h^{k+1},
$$ %\eq
which combined with the approximation result in Lemma \ref{ProjErr} leads to (\ref{pbddgenorm}) as desired.
\end{proof}

\iffalse
\begin{rem}\label{GlinBd}
For $G(s) \leq M$ with some constant $M>0$, then $G(s) \leq \max(1,M) $ and
$$
s \leq \max(G^{-1}(1),G^{-1}(M)) \leq G^{-1}(1)+\frac{G^{-1}(1)}{G^{-1}(1)-\ln \left( 1+G^{-1}(1) \right)} M.
$$
To estimate $G^{-1}(M)$ with $M\geq 1$, we denote $G^{-1}(M)=e^{s'}-1 $ with $s' \geq \ln (1+G^{-1}(1))$, then $M=e^{s'}-1-s'$ and
$$
G^{-1}(M)= M \frac{G^{-1}(M)}{M}=M\frac{e^{s'}-1}{e^{s'}-1-s'} \leq \frac{G^{-1}(1)}{G^{-1}(1)-\ln \left( 1+G^{-1}(1) \right)} M
$$
\end{rem}
\fi

 \subsection{Fully-discrete DG schemes}
%As is well known an explicit time discretization for a higher order PDE would require very small time step.  We now discuss stability properties of  different time discretizations.
Let $(u^n_h, q_h^n)$ denote the approximation to $(u_h, q_h)(\cdot, t_n)$, where $t_n=n\Delta t$ with $\Delta t$ being the time step. We consider a class of time stepping methods indexed by a parameter $\theta \in [0, 1]$:  find $(u_h^{n}, q_h^{n}) \in V_h^k \times V_h^k$ such that
for all $\phi, \ \psi \in V_h^k$
\begin{subequations}\label{FPDGFull}
\begin{align}
  \left( \frac{u^{n+1}_h - u^n_h}{\Delta t},  \phi \right) = & - A(q^{n+\theta}_h,\phi),\\
    (q^{n}_h,  \psi) = & A(u^{n}_h,\psi),
\end{align}
\end{subequations}
where $v^{n+\theta} = (1-\theta) v^n + \theta v^{n+1}$. Note that when $\theta=0$, it is the forward Euler, $\theta=1$, it is backward Euler; and $\theta=1/2$, Crank-Nicolson.

To study the stability of the DG scheme (\ref{FPDGFull}), we first recall the following estimate.
\begin{lem} (\cite[Lemma 3.2]{LW16})\label{inesti}
The following inverse inequalities hold for all $v \in V_h^k$,
\begin{equation*}
\begin{aligned}
\sum_{j=1}^N \int_{I_j} (v_x)^2 dx \leq &  \frac{(k+1)^2k(k+2)}{h^2} \|v\|^2, \\
\sum_{j=1}^N [v]^2_{j+1/2} \leq &  \frac{4(k+1)^2}{h} \|v\|^2, \\
\sum_{j=1}^N \{v_x\}^2_{j+1/2}\leq &  \frac{k^3(k+1)^2(k+2) }{h^3} \|v\|^2. \\
\end{aligned}
\end{equation*}
\end{lem}
Then, we have the following stability results.
\begin{thm}\label{FullStab}($L^2$-Stability).  For $\frac{1}{2} \leq \theta \leq 1$, the fully discrete DG scheme (\ref{FPDGFull}) is unconditionally $L^2$ stable. Moreover,
\bq\label{stab1d}
\| u_h^{n+1}\|^2 \leq \| u_h^n\|^2 -2 \Delta t \|(1-\theta)q^n_h +\theta q^{n+1}_h\|^2
\eq
holds for any $\Delta t>0$. For $0 \leq \theta < \frac{1}{2}$,
(\ref{FPDGFull}) is $L^2$ stable, i.e.,
$
\| u_h^{n+1}\| \leq \| u_h^n\|,
$
provided
\bq\label{CFL}
\Delta t  < \frac{2h^4 }{(1-2\theta)\gamma^2(k)},
\eq
where
\bq\label{GammaK}
\gamma(k) = (k+1)^2k(k+2)+ 4(k+1)^2k \sqrt{k(k+2)}.
\eq
\end{thm}

\begin{proof}
From  (\ref{FPDGFull}b) it follows
\bqs%\label{FPDGFullCom}
\int_a^b q^{n+\theta}_h \psi dx = A(u^{n+\theta}_h,\psi).
\eqs
This relation when added upon  (\ref{FPDGFull}a) with $\phi = u^{n+\theta}_h, \; \psi = q^{n+\theta}_h$ gives
%By taking $\phi = u^{n+\theta}, \; \psi = q^{n+\theta}$ in (\ref{FPDGFull}a) and (\ref{FPDGFull}b) separately, and taking the addition gives
\bq\label{stabstep}
\int_a^b \frac{u^{n+1}_h - u_h^n}{\Delta t} u_h^{n+\theta} dx + \| q_h^{n+\theta}\|^2 =0.
\eq
Using the identity
$$
u_h^{n+\theta}= \frac{1}{2}\left(u_h^{n+1}+u_h^n \right) + \left(\theta-\frac{1}{2} \right) \left(u_h^{n+1}-u_h^n \right),
$$
we rewrite  (\ref{stabstep}) as
\begin{equation}\label{FullStabF}
\begin{aligned}
 \|u^{n+1}_h\|^2-\|u_h^n\|^2 + 2\Delta t \| q_h^{n+\theta}\|^2 = (1-2\theta) \|u_h^{n+1} - u_h^n\|^2.
\end{aligned}
\end{equation}
This implies (\ref{stab1d}) if $\frac{1}{2} \leq \theta \leq 1$. %In such case DG scheme (\ref{FPDGFull}) is unconditionally stable.
If $0 \leq \theta < \frac{1}{2}$, we need to estimate the right hand side of (\ref{FullStabF}). By taking $\phi=u^{n+1}_h - u_h^n$ in (\ref{FPDGFull}a) and using Lemma \ref{inesti}, we have
\begin{equation*}\label{FullStabRhs2nd}
\begin{aligned}
\frac{1}{\Delta t }\|u_h^{n+1} - u_h^n\|^2 = & - A(q_h^{n+\theta},u_h^{n+1} - u_h^n) \\
\leq &  \left(\sum_{j=1}^N \int_{I_j} (q^{n+\theta}_{hx})^2 dx \right)^\frac{1}{2} \left(\sum_{j=1}^N \int_{I_j} (u^{n+1}_{hx} - u^n_{hx})^2 dx \right)^\frac{1}{2} \\
& + \left( \sum_{j=1}^N \{q^{n+\theta}_{hx}\}^2_{j+1/2} \right)^\frac{1}{2} \left( \sum_{j=1}^N [u_h^{n+1} - u_h^n]^2_{j+1/2} \right)^\frac{1}{2} \\
& + \left( \sum_{j=1}^N [q_h^{n+\theta}]^2_{j+1/2} \right)^\frac{1}{2} \left( \sum_{j=1}^N \{(u^{n+1}_{hx} - u^n_{hx})_{j+1/2}\}^2 \right)^\frac{1}{2} \\
\leq & \frac{(k+1)^2k(k+2)+4(k+1)^2k \sqrt{k(k+2)}}{h^2} \|q_h^{n+\theta}\| \|u_h^{n+1} - u_h^n\| \\
= & \frac{\gamma(k) }{h^2} \|q_h^{n+\theta}\| \|u_h^{n+1} - u_h^n\|,
\end{aligned}
\end{equation*}
% {\color{red} More details are needed!}\\
with  $\gamma(k)$ defined in (\ref{GammaK}).
Hence
\bqs%\label{StabDiff2nd}
\|u_h^{n+1} - u_h^n\| \leq \frac{\Delta t \gamma(k) }{h^2} \|q_h^{n+\theta}\|.
\eqs
This upon insertion into (\ref{FullStabF}) yields
\begin{equation*}
\begin{aligned}
 \|u_h^{n+1}\|^2 + \Delta t \left(2-(1-2\theta) \frac{\Delta t \gamma^2(k)}{h^4} \right) \| q_h^{n+\theta}\|^2 \leq \|u_h^n\|^2.
\end{aligned}
\end{equation*}
By (\ref{CFL}) we therefore obtain the desired stability, i.e.,   $\|u_h^{n+1}\|\leq \|u_h^n\|$.
%so that the scheme (\ref{FPDGFull}) is stable as long as
%\bq
%2-(1-2\theta) \frac{\Delta t \gamma^2(k)}{h^4}>0,
%\eq
%which gives (\ref{CFL}).
\end{proof}

The above results suggest that the semi-implicit time discretization with $\theta \in [1/2, 1]$ should be considered.  To assist  the error estimate for the fully-discrete DG scheme (\ref{FPDGFull}) with $\theta \in [1/2, 1]$,  we prepare the following lemma.
%\begin{lem}\label{anEsti}
%Let $\{a_n\}$ be a non-negative sequence satisfying
%\bq\label{ann}
%\frac{a_{n+1}^2 - a_n^2}{\tau} \leq C_1 (a_{n+1} +a_n) +C_2,
%\eq
%where $C_1, C_2$ are positive numbers, then for all $n$, we have
%\bq\label{anBnd}
%a_n \leq \frac{C_2}{C_1} G^{-1}\left(\frac{C_1}{C_2}a_0+\frac{C_1^2}{C_2} n\tau \right),
%\eq
%where
%\bq\label{GXD}
%G(s)= s-\ln \sqrt{2s +1}, \quad s>0.
%\eq
%\end{lem}
\begin{lem}\label{anEsti}
Let $\{a_n\}$ with $a_0>0$ be a non-negative sequence satisfying
\bq\label{ann}
\frac{a_{n+1}^2 - a_n^2}{\tau} \leq \alpha (a_{n+1} +a_n+1),
\eq
where $\tau>0$ and $\alpha>0$, then there exists $C=C(a_0, \alpha)$ such that
\bqs%\label{anBnd}
a_n \leq a_0+C n\tau, \quad  \forall n \geq 1.
\eqs
%where $a'_0>a_0$ is a positive a constant and
%\bq\label{GXD}
%G(s)= s-\ln \sqrt{2s +1}, \quad s>0.
%\eq
\end{lem}
\begin{proof} Define  $A_n=\max_{0\leq i\leq n} a_i$, then (\ref{ann}) remains valid for $A_n$, i.e.,
\bq\label{Ann}
\frac{A_{n+1}^2 - A_n^2}{\tau} \leq \alpha (A_{n+1} +A_n+1).
\eq
In fact, we have  $
a_n \leq A_n, \forall  n \geq 0,
$
and
$$ %\bq\label{An1Max}
A_{n+1}=\max \{ a_{n+1}, A_n\}.
$$ % \eq
If $A_{n+1}=A_n$, (\ref{Ann}) is obvious; otherwise if $A_{n+1}=a_{n+1}$,
it follows that
\bqs
\frac{A_{n+1}^2 - A_n^2}{\tau} \leq\frac{a_{n+1}^2 - a_n^2}{\tau} \leq \alpha (a_{n+1} +a_n+1) \leq \alpha (A_{n+1} +A_n+1).
\eqs
%where we have used (\ref{ann}) and (\ref{annineq}).
Rewriting (\ref{Ann}) as
$$ % \bq\label{AB}
A_{n+1} -A_n -\frac{A_{n+1}-A_n}{A_{n+1}+A_{n}+1} \leq \alpha \tau,
$$ %\eq
and using
$$
\int_{A_n}^{A_{n+1}} \frac{1}{2s+1}ds \geq \frac{A_{n+1}-A_n}{A_{n+1}+A_{n}+1},
$$
we have
$$
H(A_{n+1})-H(A_n) \leq  \alpha \tau,
$$
where $H(s)=s -\ln \sqrt{2s+1}$, and therefore
%\bq\label{GABnd}
$$
H(A_{n}) \leq H(A_0) +\alpha n\tau.
$$
Note that $H$ is increasing and convex over $[0, \infty)$, hence we have
$$ % \bq\label{aninv}
A_n \leq  H^{-1} \left(H(A_0)+\alpha n\tau \right)=  H^{-1} \left(H(a_0)+\alpha n\tau \right).
$$ %\eq
It can be verified that $H^{-1}(s)/s$ is decreasing for $s>0$. Thus,
$$
A_n \leq \frac{a_0}{H(a_0)}\left(H(a_0)+\alpha  n\tau \right) = a_0 + C n\tau, \quad C=\frac{\alpha a_0}{H(a_0)}.
$$
Going back to $a_n \leq A_n$ we prove the claimed estimate.
%$$
%C_1=\frac{Ca_0}{a_0 -\ln \sqrt{2a_0 +1}}.
%$$
\end{proof}

%\begin{lem}\label{anEsti}
%Let $\{a_n\}$ be a non-negative sequence satisfying
%\bq
%\frac{a_{n+1}^2 - a_n^2}{\tau} \leq C_1 (a_{n+1} +a_n) +C_2,
%\eq
%where $C_1, C_2$ and $\tau$ are constants,
%then for all $n$ such that $n \tau \leq T$, we have
%\bq\label{anBnd}
%a_n \leq A^*,
%\eq
%where
%\bq
%A^* = \left \{ x\in[0, +\infty) | G(x) = C_1 T +G(0)  \right \},
%\eq
%and the convex function
%\bq\label{GXD}
%G(x)= x-\frac{C_2}{2C_1} \ln \left( \frac{2 C_1}{C_2}x +1 \right).
%\eq
%\end{lem}
%\begin{proof}
%Take $A_n=\max_{0\leq i\leq n} a_i$, then
%\bq\label{anVSAn}
%a_n \leq A_n,
%\eq
%and
%\bq\label{An1Max}
%A_{n+1}=\max \{ a_{n+1}, A_n\}.
%\eq
%From (\ref{An1Max}), if $a_{n+1} \leq A_n$, we have
%\bq\label{An1Eql}
%A_{n+1} = A_n,
%\eq
%and if $a_{n+1} > A_n$, then $A_{n+1} = a_{n+1}$ and it follows
%\bq
%\frac{A_{n+1}^2 - A_n^2}{\tau} \leq \frac{a_{n+1}^2 - a_n^2}{\tau} \leq C_1 (a_{n+1} +a_n) +C_2 \leq C_1 (A_{n+1} +A_n) +C_2,
%\eq
%which gives
%\bq\label{An1Bnd}
%A_{n+1}-A_n \leq  \frac{C_2(A_{n+1}-A_n)}{C_1(A_{n+1}+A_n)+C_2} +C_1\tau \leq  \frac{C_2(A_{n+1}-A_n)}{2C_1 A_n+C_2} +C_1\tau.
%\eq
%Define a convex function $G(x)$ as in (\ref{GXD}), then we have
%\bq\label{GATay}
%G(A_{n+1})-G(A_n) \leq G'(A_n) (A_{n+1}-A_n) = A_{n+1}-A_n - \frac{C_2(A_{n+1}-A_n)}{2C_1 A_n+C_2}.
%\eq
%Then (\ref{An1Eql}) and (\ref{An1Bnd}) imply
%\bq\label{GABnd}
%G(A_{n+1})-G(A_n) \leq C_1\tau,
%\eq
%and taking summation from $0$ to $n$, we obtain
%\bq\label{GAEq}
%\ba
%G(A_{n}) \leq &  C_1 n \tau + G(A_0) \\
%\leq & C_1 T + G(A_0). \\
%\ea
%\eq
%Solving (\ref{GAEq}) along with (\ref{anVSAn}) gives (\ref{anBnd}).
%\end{proof}

\begin{thm}\label{thmErrFull}
Let $u^n_h$ be the numerical solution to the fully-discrete DG scheme (\ref{FPDGFull}) with $\frac{1}{2} \leq \theta  \leq 1$, and $u$ be the smooth solution to problem (\ref{Beam1D}), then
\bq\label{pbddgFullUL2}
\|  u(\cdot,t^n)-u^n_h(\cdot) \| \leq C \left ( h^{k+1} + (\theta- 1/2) \Delta t + (\Delta t)^2 \right),\\
\eq
where  $C$ depends on $\sup_{t \in [0,T]}|u_t(\cdot,t)|_{k+1}$, $\sup_{t \in [0,T]}|u(\cdot, t)|_{k+3}$, { $\sup_{t\in [0,T]} \| u_{tt}(\cdot, t)\|$,  $\sup_{t\in [0,T]} \|u_{ttt}(\cdot, t)\|$} and linearly on $T$, but independent of  $h, \Delta t$.
\end{thm}
\begin{proof}
Denote $u^n=u(x,t^n)$ and $q^n=q(x,t^n)$, then the consistency of the DG scheme, as given in (\ref{FPDGExact}), when evaluated at $t=t^{n+\theta}$ is
\bq\label{FPDGETn}
\begin{aligned}
(u^{n+\theta}_t, \phi) = & - A(q^{n+\theta},\phi),\\
(q^{n}, \psi) = & A(u^{n},\psi),\\
\end{aligned}
\eq
for all $\phi \in V_h^k, \psi \in V_h^k$, where $v^{n+\theta} = \theta v^{n+1}+(1-\theta) v^n$ for $v=u, q$. To proceed, we first evaluate the term $u^{n+\theta}_t$. By Taylor's expression, we have
\bqs
\ba
u^n_t= & \frac{u^{n+1}-u^n}{\Delta t}-\frac{1}{2} u^n_{tt} \Delta t-\frac{1}{2\Delta t}\int_{t^n}^{t^{n+1}} (t^{n+1}-s)^2  u_{ttt}(x,s) ds, \\
u^{n+1}_t = & \frac{u^{n+1}-u^n}{\Delta t}+\frac{1}{2} u^{n+1}_{tt} \Delta t -\frac{1}{2\Delta t}\int_{t^n}^{t^{n+1}} (t^n-s)^2 u_{ttt}(x,s) ds,
\ea
\eqs
%along with
%\bqs
%u^{n+1}_{tt} = u^n_{tt} + \int_{t^n}^{t^{n+1}} \partial_t^3 u(x,s) ds,
%\eqs  % then it follows
so that
$$% \bq\label{UNTheta}
u^{n+\theta}_t = \theta u^{n+1}_t + (1-\theta) u^n_t = \frac{u^{n+1}-u^n}{\Delta t} + F(n,x,t,\theta), \\
$$ % \eq
where
%\bq
\begin{align*}
F(n,x,t,\theta) = & u^n_{tt} \Delta t \left( \theta- \frac{1}{2}\right)  - (1-\theta) \left (\frac{1}{2\Delta t}\int_{t^n}^{t^{n+1}} (t^{n+1}-s)^2 u_{ttt}(x,s) ds \right) \\
& + \theta \left(\frac{1}{2} \Delta t \int_{t^n}^{t^{n+1}} u_{ttt}(x,s) ds -\frac{1}{2\Delta t}\int_{t^n}^{t^{n+1}} (t^n-s)^2 u_{ttt}(x,s) ds \right).
\end{align*}
%\eq
Then (\ref{FPDGETn}) becomes
$$ % \bq\label{FPDGETnEx}
\begin{aligned}
\left(\frac{u^{n+1}-u^n}{\Delta t}, \phi \right) = & - A(q^{n+\theta},\phi) - (F(n,x,t,\theta), \phi),\\
(q^{n}, \psi) = & A(u^{n},\psi),
\end{aligned}
$$ %\eq
which together with  (\ref{FPDGFull})  gives
\bq\label{FPDGETnSub}
\begin{aligned}
\left(\frac{(u^{n+1}-u_h^{n+1})-(u^n-u_h^n)}{\Delta t}, \phi \right) = & - A(q^{n+\theta}-q_h^{n+\theta},\phi) - (F(n,x,t,\theta), \phi),\\
(q^{n+\theta}-q_h^{n+\theta}, \psi) = & A(u^{n+\theta}-u_h^{n+\theta},\psi).
\end{aligned}
\eq
%where we took the combination of second equation at $t^n$ and $t^{n+1}$.\\
Denote
\bqs
\ba
e_1^n = & Pu^n - u^n_h, \quad \epsilon_1^n = Pu^n -u^n, \\
e_2^n = & Pq^n - q^n_h, \quad \epsilon_2^n = Pq^n -q^n,
\ea
\eqs
and take $\phi = e_1^{n+\theta}, \ \psi = e_2^{n+\theta}$ in (\ref{FPDGETnSub}),  upon summation and using (\ref{AA}),
%it follows
%\begin{subequations}\label{FPDGETnErr}
%\begin{align}
%\left(\frac{e_1^{n+1}-e_1^n}{\Delta t}, e_1^{n+\theta} \right) = & \left(\frac{\epsilon_1^{n+1}-\epsilon_1^n}{\Delta t}, e_1^{n+\theta} \right)- A(e_2^{n+\theta},e_1^{n+\theta}) +A(\epsilon_2^{n+\theta},e_1^{n+\theta})  - (F(n,x,t,\theta), e_1^{n+\theta}),\\
%(e_2^{n+\theta}, e_2^{n+\theta}) = &(\epsilon_2^{n+\theta}, e_2^{n+\theta}) + A(e_1^{n+\theta},e_2^{n+\theta})-A(\epsilon_1^{n+\theta},e_2^{n+\theta}).
%\end{align}
%\end{subequations}
%Taking summation of (\ref{FPDGETnErr}a) and (\ref{FPDGETnErr}b) and  using (\ref{AA}),
we obtain
\bq\label{PFDGErr}
\left(\frac{e_1^{n+1}-e_1^n}{\Delta t}, e_1^{n+\theta} \right) + (e_2^{n+\theta}, e_2^{n+\theta}) = \left(\frac{\epsilon_1^{n+1}-\epsilon_1^n}{\Delta t}, e_1^{n+\theta} \right) + (\epsilon_2^{n+\theta}, e_2^{n+\theta}) -(F(n,x,t,\theta), e_1^{n+\theta}).
\eq
Applying
$$
e_1^{n+\theta}= \frac{1}{2}\left(e_1^{n+1}+e_1^{n} \right) + \left(\theta-\frac{1}{2} \right) \left(e_1^{n+1}-e_1^{n} \right),
$$
and the Cauchy-Schwarz inequality to (\ref{PFDGErr}), it follows that
\bq\label{ErrUpBnd1}
\ba
&\frac{\|e_1^{n+1}\|^2-\|e_1^n\|^2}{2\Delta t} \leq \left( \left \| \frac{\epsilon_1^{n+1}-\epsilon_1^n}{\Delta t} \right \| + \|F(n,\cdot,t,\theta)\| \right) (\| e_1^{n+1}\| + \| e_1^n\|) + \frac{1}{2}\left(\|\epsilon_2^{n+1}\|^2 +\|\epsilon_2^{n}\|^2\right).
\ea
\eq
Recall the projection error estimate (\ref{ProjErr1D}), we have
\bq\label{EpsiBnd}
\ba
\|\epsilon_2^{n+i}\|
\leq & Ch^{k+1} | q(\cdot,t^{n+i})|_{k+1} =  C_1 h^{k+1} |u(\cdot,t^{n+i})|_{k+3},
\ea
\eq
for $i=0,1$, and along with the mean value theorem, we also have
\bq\label{PartEpsi}
\ba
\left \| \frac{\epsilon_1^{n+1}-\epsilon_1^n}{\Delta t} \right \| %= & \left \| \frac{Pu^{n+1}-Pu^n}{\Delta t} - \frac{u^{n+1}-u^n}{\Delta t}  \right \|
=   \left \| P \left( \frac{u^{n+1}-u^n}{\Delta t} \right)- \frac{u^{n+1}-u^n}{\Delta t}  \right \|
\leq  C_2h^{k+1} |u_{t}(\cdot,t^*)|_{k+1}, \\
\ea
\eq
where $t^* \in (t^n, t^{n+1})$.  As for the term involving $F$, we have
$$
\ba
| F(n,x,t,\theta) | \leq & \left| u^n_{tt} \right | \Delta t \left( \theta- \frac{1}{2}\right) + \frac{(1-\theta)}{2\Delta t}\int_{t^n}^{t^{n+1}} (t^{n+1}-s)^2 | u_{ttt}(x,s) |  ds \\
& +  \frac{\theta \Delta t }{2}\int_{t^n}^{t^{n+1}}\left| u_{ttt}(x,s) \right | ds + \frac{\theta}{2\Delta t}\int_{t^n}^{t^{n+1}} (t^n-s)^2 \left| u_{ttt}(x,s) \right | ds \\
 \leq & \left( \theta- \frac{1}{2}\right) \Delta t \sup_{t\in [0,T]} | u_{tt}(x, t)|+\left(\frac{1}{6}+\frac{\theta}{2}\right) (\Delta t)^2 \sup_{t\in [0,T]} |u_{ttt}(x, t)|,
\ea
$$
hence
\bq\label{FBnd}
\| F(n,\cdot,t,\theta) \|\leq \left( \theta- \frac{1}{2}\right) \Delta t \sup_{t\in [0,T]} \|u_{tt}(\cdot, t)\|+ (\Delta t)^2 \sup_{t\in [0,T]} \|u_{ttt}(\cdot, t)\|.
\eq
Plugging (\ref{FBnd}), (\ref{PartEpsi}) and (\ref{EpsiBnd})  into (\ref{ErrUpBnd1}) leads to
\bqs%\label{ErrUpBnd2}
\ba
\frac{\|e_1^{n+1}\|^2-\|e_1^n\|^2}{2\Delta t} \leq &  C \left ( h^{k+1} + \left( \theta- 1/2\right) \Delta t + (\Delta t)^2 \right) (\| e_1^{n+1}\| + \| e_1^n\|) + C h^{2(k+1)},
%\leq & C_1 \left ( h^{k+1} + \left( \theta- 1/2\right) \Delta t + (\Delta t)^2 \right) (\| e_1^{n+1}\| + \| e_1^n\|) \\
%& + C_2 \left ( h^{k+1} + \left( \theta- 1/2\right) \Delta t + (\Delta t)^2 \right)^2,
\ea
\eqs
where  $C$ depends on $\sup_{t\in [0,T]}|u_t(\cdot, t)|_{k+1}$, $\sup_{t\in [0,T]}|u(\cdot, t)|_{k+3}$, { $\sup_{t\in [0,T]} \| u_{tt}(\cdot, t)\|$ and $\sup_{t\in [0,T]} \| u_{ttt}(\cdot, t)\|$}.
%Denote $B^n=\frac{\|e_1^n\|}{\left ( h^{k+1} + \left| \theta- \frac{1}{2}\right| \Delta t + \left( \frac{\theta}{2}+\frac{1}{6} \right)\Delta t)^2 \right)}\geq 0$, then (\ref{ErrUpBnd2}) becomes
%\bq\label{ErrUpBnd1Sc}
%\ba
%&\frac{(B^{n+1})^2-(B^n)^2}{2\Delta t} \leq C_1 ( B^{n+1} + B^n) + C_2.
%\ea
%\eq

Set $a_n=\frac{\|e_1^n\|}{ h^{k+1}}$, $\tau= 2\Delta t$, then
$a_n$ satisfies (\ref{ann}) with
$$\alpha=Ch^{-(k+1)}\left ( h^{k+1} + \left( \theta- 1/2\right) \Delta t + (\Delta t)^2 \right).
$$
 Note that $e_1^0= Pu_0-u_h^0$ and
$
\|e_1^0 \| \leq \| Pu_0-u_0\|+\|u_0-u_h^0 \| \leq C_0 h^{k+1},
$
we thus take  $a_0= C_0$. By Lemma \ref{anEsti} we have
\bqs%\label{ProjNum}
\ba
\|e_1^n\| \leq h^{k+1} \left(C_0 +  \frac{C_0\alpha}{H(C_0)} n\tau \right)  \leq C(1+T) \left ( h^{k+1} + \left( \theta- 1/2\right) \Delta t + (\Delta t)^2 \right),
\ea
\eqs
%\bq\label{ProjNum}
%\ba
%\|e_1^n\| \leq & \frac{C_2}{C_1} G^{-1}\left(\frac{C_1}{C_2}\|e_1^0\| +\frac{C_1^2}{C_2}T  \right) \left ( h^{k+1} + \left( \theta- 1/2\right) \Delta t + (\Delta t)^2 \right).
%\ea
%\eq
which combined with the projection error (\ref{ProjErr1D}) leads to (\ref{pbddgFullUL2}) as desired.
\end{proof}

\subsection{Algorithm}
The details related to the implementation of scheme (\ref{FPDGFull}) with $\theta \in [1/2, 1]$ is summarized in the following algorithm.
\begin{itemize}
  \item Step 1 (Initialization) from the given initial data $u_0(x)$,
  \begin{enumerate}
    \item generate $u_h^0:=u_h(x,0) \in V_h^k$ from  the piecewise $L^2$ projection (\ref{LocL2Proj}), and
    \item further obtain $q_h^0$ from solving (\ref{FPDGFull}b).
  \end{enumerate}
  \item Step 2 (Evolution) obtain $u^{n+1}_h, \ q^{n+1}_h$ by solving (\ref{FPDGFull}) through the following form:
 % \begin{enumerate}
    %\item set the parameter $\theta=1/2$.
   % \item generate the following linear system %by taking test functions $\phi, \ \psi \in V_h^k$ in the following system
    \begin{subequations}\label{FPDGFullAlg}
    \begin{align}
   \frac{1}{\Delta t}  (u^{n+1}_h,  \phi) +  \theta  A(q^{n+1}_h,\phi)= &  \frac{1}{\Delta t}  ( u^{n}_h,  \phi ) - (1-\theta) A(q^{n}_h,\phi),\\
 \theta A(u^{n+1}_h,\psi)- \theta  ( q^{n+1}_h,  \psi) = 0.
    \end{align}
    \end{subequations}
  %  (Here, (\ref{FPDGFullAlg}b) is obtained by multiplying (\ref{FPDGFull}b) by $-\theta$ at $n+1$ level.)
    %\item
  %\end{enumerate}
\end{itemize}
\begin{rem} The advantage of using (\ref{FPDGFullAlg}) is that its  coefficient matrix is symmetric, hence more efficient linear system solvers, such as the ILU preconditioner + FGMRES (see e.g.,  \cite{S93}), ILU preconditioner + Bicgstab (see e.g.,  \cite{CMZ16}).
can be used.
\end{rem}

\section{The DG scheme in multi-dimensions} In this section we present DG schemes in multi-dimensional setting.   Without loss of generality,  we describe our DG scheme and prove the optimal error estimates in two dimension ($d=2$); The analysis depending  on the tensor product of polynomials  can be easily extended to higher dimensions.  Hence, from now on we shall restrict ourselves mainly to the following two-dimensional problem
\begin{subequations}\label{3.1}
\begin{align}
& u_t=-(\partial_x^2 +\partial_y^2)^2 u, \quad (x, y)\in \Omega, \; t>0, \\
& u(x, y, 0)=u_0(x, y), \quad (x, y)\in \Omega
\end{align}
\end{subequations}
again with periodic boundary conditions.

We partition $\Omega$ by rectangular meshes
\begin{equation*}
    \Omega=\sum^{N,M}_{i,j}I_{i,j}, \quad  I_{i,j}=[x_{i-\frac{1}{2}}, x_{j+\frac{1}{2}}]\times[y_{j-\frac{1}{2}}, y_{j+\frac{1}{2}}].
\end{equation*}
For simplicity we assume we have a uniform rectangular mesh with
$\Delta x=x_{i+1/2}-x_{i-1/2}, \Delta y=y_{j+1/2}-y_{j-1/2}$.
Let
$$
Q_h=\{v\in L^2(\Omega): \quad v|_{I_{i,j}}\in Q^k(I_{i,j})\},
$$
where $Q^k(K)$ denotes the space of tensor-product polynomials of degree at most $k$ in each variable defined on $K$. No continuity is assumed across cell boundaries.

The semi-discrete DG approximations $(u_h, q_h) \in Q_h \times Q_h$ of (\ref{3.1}) are defined through the reformulation of form (\ref{BeamSys}) such that for all admissible test functions $\phi, \ \psi \in Q_h$ and all  $I_{i,j}$
\begin{equation}\label{FPDGCell2D}
\begin{aligned}
    \iint_{I_{i, j}} u_{ht} \phi dxdy = & - \iint_{I_{i,j}} \nabla q_h \cdot \nabla\phi dxdy
    + \int_{y_{j-1/2}}^{y_{j+1/2}} \left(\{q_{hx}\} \phi  + (q_h - \{q_h\}) \phi_x \right) \Big|_{x_{i-1/2}}^{x_{i+1/2}} dy \\
    & +\int_{x_{i-1/2}}^{x_{i+1/2}} \left(\{q_{hy}\} \phi +  (q_h-\{q_h\})\phi_y\right) \Big|_{y_{j-1/2}}^{y_{j+1/2}} dx, \\
   % & + \alpha \left( - \iint_{I_{j,k}} \nabla u \cdot \nabla\phi dxdy+ \int_{I_k}\widehat{(u_x)} \phi |_{\partial I_j} + (u - \widehat{u})\phi_x|_{\partial I_j} dy \right) \\
  %  & + \alpha \left(  \int_{I_j}\widehat{(u_y)} \phi |_{\partial I_k} + (u - \widehat{u})\phi_y|_{\partial I_k} dx \right)+\iint_{I_{j,k}} f \phi dxdy \\
    \iint_{I_{i,j}} q_h \psi dxdy
    = & \iint_{I_{i,j}} \nabla u_h \cdot \nabla \psi dxdy - \int_{y_{j-1/2}}^{y_{j+1/2}} \left(\{u_{hx}\} \psi  + (u_h-\{u_h\})\psi_x
    \right)  \Big|_{x_{i-1/2}}^{x_{i+1/2}} dy \\
    &- \int_{x_{i-1/2}}^{x_{i+1/2}} \left( \{u_{hy}\} \psi  + (u_h-\{u_h\})\psi_y \right)\Big|_{y_{j-1/2}}^{y_{j+1/2}} dx, \\
\end{aligned}
\end{equation}
where
\begin{equation*}
\begin{aligned}
& v\Big|_{x_{i-1/2}}^{x_{i+1/2}} =v(x_{i+1/2}^-, y) - v(x_{i-1/2}^+, y), \\
& v\Big|_{y_{j-1/2}}^{y_{j+1/2}} =v(x, y_{j+1/2}^-) - v(x, y_{j-1/2}^+),\\
& \{v\}\Big|_{x_{i+1/2}} =  \frac{1}{2} \left(v(x_{i+1/2}^-, y) + v(x_{i+1/2}^+, y)\right),
\\
& \{v\}\Big|_{y_{j+1/2}}=\frac{1}{2} \left(v(x, y_{j+1/2}^-) +v(x, y_{j+1/2}^+)\right).
\end{aligned}
\end{equation*}
The initial data for $u_h$ is also taken as the piecewise $L^2$ projection of $u_0$, that is $u_h(x, y, 0) \in Q_h$ such that
\begin{align*}
\iint_{\Omega} (u_0(x,y)-u_h(x, y, 0))\phi(x,y)dxdy=0, \quad \forall \phi\in Q_h.
%\iint_{I_{i,j}} (-\Delta u_0(x,y)-q_h(0, x, y))\psi(x,y)dxdy=0, \quad \forall \psi\in Q_h.
\end{align*}
%Further $q_h(x, y, 0)$ can be obtained from the second equation of  (\ref{FPDGCell2D}).
\subsection{Stability and a priori error estimates} In order to check the stability of  the above scheme, we sum (\ref{FPDGCell2D}) over all computaitonal cells to obtain
\begin{subequations}\label{FPDG1D+}
\begin{align}
(u_{ht}, \phi) = &- A(q_h,\phi),\\
(q_h, \psi) = & A(u_h,\psi),
\end{align}
\end{subequations}
where $(\cdot, \cdot)$ denotes the inner product of two functions over $\Omega$, and the bilinear functional
\bq
\label{bilinear+}
\ba
A(w,v)= & \sum_{i,j=1}^{N,M}\iint_{I_{i,j}} \nabla w \cdot \nabla v dxdy
 +\sum_{i, j=1}^{N, M} \int_{y_{j-1/2}}^{y_{j+1/2}}
 \left( \{w_{x}\} [v]  + \{v_x\} [w] \right)_{x_{i+1/2}}dy  \\
& +\sum_{i, j=1}^{N, M} \int_{x_{i-1/2}}^{x_{i+1/2}}  \left(\{w_{y}\}[v] +\{v_y\}[w] \right)_{y_{j+1/2}}dx.
\ea
\eq
For scheme (\ref{FPDGCell2D})  the following stability result holds.
\begin{thm}\label{SemiStab2D} ($L^2$-Stability). The numerical solution $u_h$ to (\ref{FPDG1D+}) satisfies
\bqs
\frac{1}{2}\frac{d}{dt} \iint_{\Omega} |u_h|^2dxdy=-\iint_{\Omega} q_h^2dxdy \leq 0.
\eqs
\end{thm}
% \subsection{A priori error estimates}
 In order to obtain the error estimate for  DG scheme (\ref{FPDG1D+}) on rectangular meshes, we follow \cite{L15} extending the one-dimensional projection to multi-dimension by taking a tensor product of $2$ one-dimensional projections as
\bqs%\label{mdproj}
\Pi w= P^{(x)} \otimes P^{(y)} w,
\eqs
where the superscripts indicate the application of one-dimensional projection operator.

We recall the following result established in \cite{LHLY17}.
\begin{lem}\label{opbnd}
For $k \geq 1$ and $\eta \in Q_h$, the linear functional $w \rightarrow A(\Pi w-w, \eta)$ is continuous on $H^{k+2}(\Omega)$ and
\begin{align*}
& |A(\Pi w-w, \eta)|\leq C h^{k+2}|w|_{k+2} \|\eta\|,\\
& \|\Pi w-w\| \leq Ch^{k+1}|w|_{k+1},
\end{align*}
where $C$ is a constant independent of $h$.
\end{lem}
We are now ready to state the a priori error estimate result for the two-dimensional case.
 \begin{thm}\label{thmUL2} Let $u_h$ be the numerical solution to the DG scheme (\ref{FPDGCell2D}) and $u$ be the smooth solution to problem (\ref{3.1}),  then
\bq\label{pbddgenormMD}
\|  u(\cdot, t)-u_h(\cdot, t) \| \leq Ch^{k+1},
\eq
for $0 \leq t \leq T$, where $C$ depends on $\sup_{t\in [0, T]}\|u_t(\cdot, t)\|_{k+1}$, and linearly on $T$,  but independent of $h$.
\end{thm}

\begin{proof}
By consistency of the DG scheme (\ref{FPDG1D+}),  we have
\bqs%\label{FPDGExactMD}
\begin{aligned}
(u_t, \phi) = & - A(q,\phi), \quad \forall \phi \in Q_h,\\
(q, \psi) = & A(u,\psi), \quad \forall \psi \in Q_h,\\
\end{aligned}
\eqs
where $u$ is the exact solution to (\ref{3.1}) with $q=-\Delta u$. Upon  subtraction of this from (\ref{FPDG1D+}), we have
\bq\label{FPDGErrSysMD}
\begin{aligned}
((u-u_h)_t, \phi) = & - A(q-q_h,\phi),\\
(q-q_h, \psi) = & A(u-u_h,\psi).
\end{aligned}
\eq
Denote
\bqs
\begin{aligned}
e_1= & \Pi u-u_h, \quad \epsilon_1= \Pi u-u, \\
e_2= & \Pi q-q_h, \quad \epsilon_2= \Pi q-q, \\
\end{aligned}
\eqs
take $\phi=e_1$ and $\psi=e_2$ in (\ref{FPDGErrSysMD}),
%we obtain
%\begin{subequations}\label{FPDGErrPlugMD}
%\begin{align}
%(\partial_t e_1, e_1) = & (\partial_t \epsilon_1, e_1) + A(\epsilon_2, e_1) - A(e_2,e_1),\\
%(e_2, e_2) = & (\epsilon_2, e_2)- A(\epsilon_1, e_2) + A(e_1,e_2).
%\end{align}
%\end{subequations}
%Summation of (\ref{FPDGErrPlugMD}a) and (\ref{FPDGErrPlugMD}b) gives
to obtain
\bqs
\ba
\frac{1}{2} \frac{d}{dt}\|e_1\|^2 + \|e_2\|^2 = ( \epsilon_{1t}, e_1) + (\epsilon_2, e_2) + A(\epsilon_2, e_1)- A(\epsilon_1, e_2).
\ea
\eqs
By  Schwartz's inequality and Lemma \ref{opbnd}  we have
\bqs
\ba
\frac{1}{2} \frac{d}{dt}\|e_1\|^2 + \|e_2\|^2 & \leq   \|\epsilon_{1t}\|\| e_1\| + \|\epsilon_2\|\|e_2\|+Ch^{k+2}\left(|q|_{k+2}\|e_1\|+|u|_{k+2}\|e_2\|\right) \\
& \leq C \left( |u_t|_{k+1}+|q|_{k+2}h  \right) h^{k+1}\|e_1\| + C\left( |q|_{k+1} +|u|_{k+2}h\right)h^{k+1} \|e_2\| \\
&\leq C_1 h^{k+1} (\|e_1\| + h^{k+1}) + \|e_2\|^2,
\ea
\eqs
where { $C_1=\max \{C\sup_{t\in [0,T]}\left( |u_t|_{k+1}+|q|_{k+2}h  \right), \frac{C^2}{4}\sup_{t\in [0,T]}\left( |q|_{k+1} +|u|_{k+2}h\right)^2\}$.}
%where we have used the estimates in  Lemma \ref{opbnd}, including
%\bq\label{ProjErrAppMD}
%\begin{aligned}
%\|\partial_t \epsilon_1\|\leq C h^{k+1} |u_t|_{k+1},\quad
%\|\epsilon_2\| \leq & Ch^{k+1} |q|_{k+1};
%\end{aligned}
%\eq
%and
%$C_1=C \left( |\partial_t u|_{k+1}+|q|_{k+2}h  \right)$ and $C_2=\frac{C^2}{4}\left( |q|_{k+1} +|u|_{k+2}h\right)^2$.
Following the same analysis as that in Theorem \ref{ThmL2Err}, we obtain the estimate for $\|e_1\|$, further (\ref{pbddgenormMD}) as desired.
\end{proof}
\subsection{ Time discretization}
Let $(u^n_h, q_h^n)$ denote the approximation to $(u_h, q_h)(\cdot, t_n)$, where $t_n=n\Delta t$ with $\Delta t$ being the time step. We consider a class of time stepping methods in terms of a parameter $\theta \in [1/2, 1]$:  find $(u_h^{n}, q_h^{n}) \in Q_h \times Q_h$ such that
for all $\phi, \ \psi \in Q_h$,
\begin{subequations}\label{FPDGFull+}
\begin{align}
    \left(\frac{u_h^{n+1} - u^n_h}{\Delta t},  \phi \right) = & - A(q_h^{n+\theta},\phi),\\
   ( q_h^{n},  \psi ) = & A(u_h^{n},\psi),
\end{align}
\end{subequations}
where the notation $v^{n+\theta} := (1-\theta) v^n + \theta v^{n+1}$ is used.  Similar to the one-dimensional case, we have the following stability result.
\begin{thm}\label{FullStab+}($L^2$-Stability).
For $\frac{1}{2} \leq \theta \leq 1$, the fully discrete DG scheme (\ref{FPDGFull+})  is unconditionally $L^2$ stable. Moreover,
$$
\|u_h^{n+1}\|^2 \leq \|u_h^n\|^2 -2\Delta t \| (1-\theta) q^n_h + \theta q^{n+1}_h \|^2
$$
holds for any $\Delta t>0$.
\end{thm}

In virtue of Lemma \ref{opbnd} and the techniques in Theorem \ref{thmErrFull}, we can obtain the error estimates for the full DG scheme (\ref{FPDGFull+})
on rectangular meshes without additional difficulty.
\begin{thm}\label{thmErrFull+}
Let $u^n_h$ be the numerical solution to the fully-discrete DG scheme (\ref{FPDGFull+}) with $\frac{1}{2} \leq \theta  \leq 1$, and $u$ be the smooth solution to problem (\ref{3.1}), then
\bqs
\|  u(x,t^n)-u^n_h \| \leq C \left ( h^{k+1} + \left( \theta- 1/2\right) \Delta t + (\Delta t)^2 \right),\\
\eqs
where $C$ depends on $\sup_{t \in [0,T]}\|u_t\|_{k+1}$, $\sup_{t\in [0,T]} \|u_{tt}(\cdot, t)\|$,  $\sup_{t\in [0,T]} \|u_{ttt}(\cdot, t)\|$ and linearly on $T$, but independent of $h, \Delta t$.
\end{thm}

%\begin{thm}
%Let $(u^n, q^n)$ be the numerical solution to the fully-discrete penalty-free DG scheme (\ref{FPDGCellFull}) satisfying the condition in Theorem \ref{FullStab}, and $U$ be the smooth solution to problem (\ref{Beam}) with $Q=-\Delta U$,  then the error estimate in $L^2$ norm
%\begin{subequations}\label{pbddgFullUL2}
%\begin{align}
%\|  U(x,t^n)-u^n \| \leq & C \left ( h^{k+1} + \left| \theta- \frac{1}{2}\right| \Delta t + \left( \frac{\theta}{2}+\frac{1}{6} \right)(\Delta t)^2 \right),\\
%\|  Q(x,t^n)-q^n \| \leq & C \left ( h^{k+1} + \left| \theta- \frac{1}{2}\right| \Delta t + \left( \frac{\theta}{2}+\frac{1}{6} \right)(\Delta t)^2 \right),
%\end{align}
%\end{subequations}
%%\bq
%%\label{pbddgFullQL2}
%%\|  Q-q^{n+1} \| \leq Ch^{k+1},
%%\eq
%where $C$ depends is independent of $h, \Delta t$. \\
%\end{thm}
\section{Extensions}
In this section, we discuss several extensions regarding the more general equation, non-homogeneous boundary conditions,   and an application to a nonlinear problem.
\subsection{General 4th order linear operaor}
We consider the general 4th order time-dependent PDEs of form (\ref{Beam}), with
\begin{equation}\label{opgen}
Lu=\sum_{m=0}^2 a_m\partial_x^{2m}u.
\end{equation}
It is known that the initial boundary value problem with periodic boundary conditions is well-posed  \cite{GKO95} if and only if there exists a constant $K$ such that
$$
a_0-a_1\xi^2 +a_2 \xi^4 \leq M
$$
holds for any real number $\xi$.  Hence, the problem is well-posed if $a_2<0$, accordingly we have $M=a_0-\frac{a_1^2}{4a_2}$ for $a_1 \leq 0$ and
$M=a_0$ for $a_1> 0$.  The case of  more interest is $a_1 \leq 0$, and will be kept in mind in the following discussion, though the scheme can also be used for  $a_1>0$.

We construct our DG scheme based on the following reformulation
\begin{equation}\label{BeamSys1D+}
\left \{
\begin{array}{rl}
    u_t = & \sqrt{-a_2} \left( \partial_x^2 +\frac{a_{1}}{2a_2} \right)q +f,\\
    q = & -\sqrt{-a_2} \left( \partial_x^2 +\frac{a_{1}}{2a_2} \right)u,
\end{array}
\right.
\end{equation}
where
$$
f =Mu,  \quad M:=\left(a_0-\frac{a_{1}^2}{4a_2} \right).
$$
The corresponding DG scheme may be given by
\begin{subequations}\label{FPDG1D+Ex}
\begin{align}
(u_{ht}, \phi) = &- \tilde A(q_h,\phi) +M(u_h, \phi),\\
(q_h, \psi) = & \tilde A(u_h,\psi),
\end{align}
\end{subequations}
where
$$
\tilde A(w,v)=\sqrt{-a_2} A(w, v)+\frac{a_1}{2\sqrt{-a_2}} (w, v)
$$
with $A(\cdot, \cdot)$ defined in (\ref{bilinear}).
%This semi-discrete DG scheme is stable in the sense that the scheme still preserves the stability inequality
%$$
%\|u(\cdot, t)\| \leq \|u(\cdot, 0)\|e^{Kt},
%$$
%which holds for the continuous case.
Such semi-discrete DG scheme can be shown $L^2$ stable, and optimally convergent. The result is summarized in the following.

%\begin{thm}\label{SemiStab+} ($L^2$-Stability). The numerical solution $u_h$ to (\ref{FPDG1D+Ex}) satisfies  the following stability property:
%\bqs
%\|u_h(\cdot, t)\| \leq \|u_h(\cdot, 0)\|e^{Kt}, \quad K=a_0-\frac{a_1^2}{4a_2}.
%\eqs
%\end{thm}
\begin{thm}\label{ThmL2ErrGen} Let $u_h$ be the numerical solution to (\ref{FPDG1D+Ex}), then
\bqs
\|u_h(\cdot, t)\| \leq \|u_h(\cdot, 0)\|e^{Mt}, \quad M=a_0-\frac{a_1^2}{4a_2}.
\eqs
Assume that the exact solution $u$ to problem (\ref{Beam}) with operator $L$ defined in (\ref{opgen}) is smooth, then
\bq\label{pbddgenormgen}
\|  u_h(\cdot, t)-u(\cdot, t) \| \leq C h^{k+1}, \quad 0\leq t\leq T,
\eq
where $C$ depends on $\sup_{t \in [0,T]}|u_t(\cdot, t)|_{k+1}$, $\sup_{t \in [0,T]}|u(\cdot, t)|_{k+3}, \ \sup_{t\in[0,T]}|u(\cdot, t)|_{k+1}$  and $T$,  but independent of $h$.
\end{thm}
\begin{proof} Firstly, the stability result follows from
$$
\frac{1}{2}\frac{d}{dt} \|u_h\|^2 +\|q_h\|^2 =M\|u_h\|^2,
$$
which is obtained by adding two equations in  (\ref{FPDG1D+Ex}) with   $\phi=u_h$ and $\psi=q_h$. Here the terms involving $A(\cdot, \cdot)$ cancel out due to the symmetry property.

We proceed to carry out the error estimate. The consistency of the DG method (\ref{FPDG1D+Ex}) ensures that the exact solution $u$ and $q$ of (\ref{BeamSys1D+}) also satisfy
\bq\label{FPDGExExact}
\begin{aligned}
(u_t, \phi) = & -\tilde A(q,\phi)+M(u, \phi),\\
(q, \psi) = & \tilde A(u,\psi),\\
\end{aligned}
\eq
for all $\phi \in V_h^k, \psi \in V_h^k$.
Subtracting (\ref{FPDG1D+Ex}) from (\ref{FPDGExExact}), we obtain the error  system
\bq\label{FPDGErrSysgen}
\begin{aligned}
((u-u_h)_t, \phi) = & - \tilde A(q-q_h,\phi)+M(u-u_h, \phi),\\
(q-q_h, \psi) = & \tilde A(u-u_h,\psi).
\end{aligned}
\eq
Denote
\begin{equation*}
\begin{aligned}
e_1= & Pu-u_h, \quad \epsilon_1= Pu-u, \\
e_2= & Pq-q_h, \quad \epsilon_2= Pq-q, \\
\end{aligned}
\end{equation*}
and take $\phi=e_1, \psi=e_2$ in (\ref{FPDGErrSysgen}),
%we obtain
%\begin{subequations}\label{FPDGErrPluggen}
%\begin{align}
%(\partial_t e_1, e_1) = & (\partial_t \epsilon_1, e_1) + \tilde A(\epsilon_2, e_1) - \tilde A(e_2,e_1)-K(\epsilon_1, e_1)+K(e_1, e_1),\\
%(e_2, e_2) = & (\epsilon_2, e_2)- \tilde A(\epsilon_1, e_2) + \tilde A(e_1,e_2).
%\end{align}
%\end{subequations}
%Summation of (\ref{FPDGErrPluggen}a) and  (\ref{FPDGErrPluggen}b) gives
 upon summation,  we obtain
\bq\label{ErrFormulagen}
\ba
\frac{1}{2} \frac{d}{dt}\|e_1\|^2 + \|e_2\|^2 = &M\|e_1\|^2 -M(\epsilon_1, e_1)+ (\epsilon_{1t}, e_1) + (\epsilon_2, e_2) +\tilde A(\epsilon_2, e_1)- \tilde A(\epsilon_1, e_2).
\ea
\eq
By property (\ref{AA}) of the projection,  we have
\begin{equation*}
\ba
\tilde A(\epsilon_2, e_1)=& \sqrt{-a_2} A(\epsilon_2, e_1) +\frac{a_1}{2\sqrt{-a_2}}(\epsilon_2, e_1)=\frac{a_1}{2\sqrt{-a_2}}(\epsilon_2, e_1), \\
\tilde A(\epsilon_1, e_2)=& \sqrt{-a_2} A(\epsilon_1, e_2)+\frac{a_1}{2\sqrt{-a_2}}(\epsilon_1, e_2)=\frac{a_1}{2\sqrt{-a_2}}(\epsilon_1, e_2).
\ea
\end{equation*}
These when inserted into (\ref{ErrFormulagen}) upon further bounding terms on the right hand side gives
%\bq
%\label{ErrFormulaSfgen}
%\frac{1}{2} \frac{d}{dt}\|e_1\|^2 + \|e_2\|^2 \leq  \|\partial_t \epsilon_1\|\| e_1\| + \frac{1}{4}\|\epsilon_2\|^2 +\|e_2\|^2.
%\eq
%%by Cauchy-Schwarz inequality the first term of the right hand side of (\ref{ErrFormulaSf}) becomes
%%\bq\label{ErrFormulaSf1}
%%|(\partial_t \epsilon_1, e_1)| \leq \|\partial_t \epsilon_1\| \|e_1\|.
%%\eq
%%For the second term, we have the following estimate
%%\bq\label{ErrFormulaSf2}
%%|(\epsilon_2, e_2)| \leq \frac{1}{4} \|\epsilon_2\|^2 + \|e_2\|^2,
%%\eq
%Hence
\bqs%\label{ErrIneqgen}
\ba
\frac{1}{2} \frac{d}{dt}\|e_1\|^2 \leq & M\|e_1\|^2+ \left(M\|\epsilon_1\| + \| \epsilon_{1t}\|+\frac{|a_1|}{2\sqrt{-a_2}}\|\epsilon_2\| \right) \|e_1\| + \frac{1}{4} \left( \|\epsilon_2\|+\frac{|a_1|}{2\sqrt{-a_2}}\|\epsilon_1\| \right)^2 \\
%\leq & \frac{d}{dt}\|\epsilon_1\| \|e_1\| + \frac{1}{4} \|\epsilon_2\|^2 \\
%\leq & C_1 h^{k+1} \|e_1\|  + C_1 h^{2(k+1)},\\
\leq & C( \|e_1\|^2+ h^{2(k+1)}),
\ea
\eqs
%By property (\ref{ProjErr1D}), we have
%\bq\label{ProjErrAppgen}
%\begin{aligned}
%\|\epsilon_1 \|\leq & C |u|_{k+1}h^{k+1},\\
%\|\partial_t \epsilon_1 \|\leq & C |u_t|_{k+1}h^{k+1},\\
%\|\epsilon_2\| \leq & C |q|_{k+1}h^{k+1} \leq C |u|_{k+3}h^{k+1}. %  = C_0h^{k+1} \sum_{j=1}^N |U|_{k+3, I_j}^2
%\end{aligned}
%\eq
%, where $C_0$ is a constant.\\
%Then, plugging (\ref{ProjErrAppgen}) into (\ref{ErrIneqgen}), we have
%\bq\label{ErrIneq1gen}
%\ba
%\frac{1}{2} \frac{d}{dt}\|e_1\|^2 \leq & \|\partial_t \epsilon_1\| \|e_1\| + \frac{1}{4} \|\epsilon_2\|^2 \\
%\frac{1}{2} \frac{d}{dt}\|e_1\|^2 \leq & K\|e_1\|^2+C h^{k+1} \|e_1\|  + C h^{2(k+1)},\\
%\ea
%\eq
where  property (\ref{ProjErr1D}) has been used,  and $C$ depends on $M$, $\sup_{t\in[0,T]}|u_t|_{k+1}$,  $\sup_{0\leq t\leq T}|u|_{k+3}$, $\sup_{t\in[0,T]}|u|_{k+1}$, independent of $h$.
By Grownwall's inequality we have
\bqs%\label{ErrProjDGgen}
\|e_1(\cdot, t)\|^2  \leq e^{2CT}(\|e_1(\cdot, 0)\|^2+ h^{2k+2}), \quad 0\leq t \leq T,
\eqs
which together with the initial error $\|e_1(\cdot, 0)\| \leq C h^{k+1}$ yields
\bqs%\label{ErrProjDGgen}
\|e_1(\cdot, t)\|  \leq \sqrt{C^2+1} e^{CT} h^{k+1}, \quad 0\leq t \leq T.
\eqs
This  when combined with the approximation result in Lemma \ref{ProjErr} leads to (\ref{pbddgenormgen}) as desired.
\iffalse
\tr{By Grownwall's inequality we have
\bqs%\label{ErrProjDGgen}
\|e_1(\cdot, t)\|  \leq \sqrt{2C} e^{CT} h^{k+1}, \quad 0\leq t \leq T,
\eqs}
\tb{
Note that $\|e_1(\cdot, 0)\| \leq C h^{k+1}$ for the same reason as (\ref{B0bnd}), then we have
\bqs%\label{ErrProjDGgen}
\|e_1(\cdot, t)\|  \leq \sqrt{C^2+1} e^{CT} h^{k+1}, \quad 0\leq t \leq T,
\eqs}
which combined with the approximation result in Lemma \ref{ProjErr} leads to (\ref{pbddgenormgen}) as desired.
\fi
\end{proof}
In a similar fashion, we consider the 2D operator
\begin{equation}\label{2DopGen}
Lu=\sum_{m=0}^2 a_m\Delta^{m}u,
\end{equation}
with $a_2<0$,   for which we have the following reformulation
\begin{equation*}%\label{BeamSys2D+}
\left \{
\begin{array}{rl}
    u_t = & \sqrt{-a_2} \left( \Delta +\frac{a_1}{2a_2} \right)q +Mu,\\
    q = & -\sqrt{-a_2} \left(\Delta +\frac{a_1}{2a_2} \right)u.
\end{array}
\right.
\end{equation*}
The corresponding DG scheme becomes
\begin{subequations}\label{FPDG1D++Ex}
\begin{align}
(u_{ht}, \phi) = &- \tilde A(q_h,\phi) +(f(u_h), \phi),\\
(q_h, \psi) = & \tilde A(u_h,\psi),
\end{align}
\end{subequations}
where $f(u)=Mu$, and  the bilinear functional for 2D rectangular meshes becomes
$$
\tilde A(w,v)= \sqrt{-a_2} A(w, v)+\frac{a_1}{2\sqrt{-a_2}} (w, v),
$$
with $A(\cdot, \cdot)$ defined in (\ref{bilinear+}). For DG scheme (\ref{FPDG1D++Ex}),  we have  the following result.
%\begin{thm}\label{SemiStab+2D} ($L^2$-Stability). The numerical solutions $u_h$ to (\ref{FPDG1D++Ex}) have the following stability property:
%\end{thm}

\begin{thm}\label{ThmL2ErrGen2D} Let $u_h$ be the numerical solution to (\ref{FPDG1D++Ex}), then
$$
\|u_h(\cdot, t)\| \leq \|u_h(\cdot, 0)\|e^{Mt}.
$$
Assume that the exact solution $u$ to problem (\ref{Beam}) with operator $L$ defined in (\ref{2DopGen}) is smooth, then
$$
\|  u_h(\cdot, t)-u(\cdot, t) \| \leq C h^{k+1}, \quad 0\leq t\leq T,
$$
where $C$ depends on $\sup_{t \in [0,T]}\|u_t(\cdot, t)\|_{k+1}$, $\sup_{t \in [0,T]}|u(\cdot, t)|_{k+3}, \ \sup_{t\in[0,T]}|u(\cdot, t)|_{k+1}$  and $T$,  but  independent of $h$.
\end{thm}

\subsection{Non-periodic boundary conditions} As is known if one of the following homogeneous boundary conditions is imposed,
$$
u=\partial_\nu u=0;  \quad u=\Delta u=0; \quad \partial_\nu u=\partial_\nu \Delta u=0, \quad x\in \partial \Omega,
$$
where $\nu$ stands for the outward normal direction to the boundary $\partial \Omega$, then the problem
$$
u_t=-\Delta^2 u, \quad u(x, 0)=u_0(x), x\in \Omega, \quad t>0
$$
is also well-posed, and $\|u(\cdot, t)\|\leq \|u_0(\cdot)\|$ holds for $t>0$.  In practice, the boundary conditions are often non-homogeneous, for example, the above three types of boundary conditions can have the form
$$
(i) \; u=g_1, \partial_\nu u=g_2;  \quad (ii)\; u= g_1, \Delta u=g_3; \quad (iii)\; \partial_\nu u=g_2, \partial_\nu \Delta u=g_4, \quad x\in \partial \Omega,
$$
where $g_i$ are given, and can be different in these three cases. The first two may be called ``generalized Dirichlet conditions" of the first and second kind, respectively, and the third one may be called  ``generalized Neumann condition". There is no restriction to the use of mixed types of boundary conditions.
% In some articles, the ``generalized Dirichlet conditions" of the first kind is also called ``clamped" boundary condition, and the second kind is called ``hinged" boundary conditions.

Let $K$ be a computation cell such that $\partial \Omega\cap K$ is not empty, with $\nu$ still denoting the outward normal direction of $\partial \Omega\cap K$. We also denote the set of all boundary edges of $\partial \Omega\cap K$ by $\Gamma$, which is a union of all boundary edges in 2D case, and $\{x_{1/2}=a, b=x_{N+1/2}\}$ in one-dimensional case.
%We also denote by $\mathcal{E}_h^I$ the set of all internal edges and by $\mathcal{E}_h^\partial$ the set of all boundary edges of $\partial \Omega\cap K$.
%Let $K$ be a computation cell  with $\nu$ still denoting the outward normal direction of $K$. We also denote by $\mathcal{E}_h^I$ the set of all internal edges and by $\mathcal{E}_h^\partial$ the set of all boundary edges.
We can then define the boundary fluxes for all edges $e \in \Gamma$
\iffalse
\bq\label{fluxBD}
\ba
\widehat{u_h} = g_1 \quad \text{or } \; &  \widehat{u} = u_h, \\
\widehat{\partial_\nu u_h} = g_2 \quad \text{or } \; & \widehat{\partial_\nu u_h} =\frac{\beta_0}{h} [u_h]+\partial_\nu u_h, \\
\widehat{q_h} = -g_3 \quad \text{or }  \; & \widehat{q_h} = q_h, \\
\widehat{\partial_\nu q_h} =  -g_4 \quad \text{or }  \; & \widehat{\partial_\nu q} =\frac{\beta_0}{h} [q_h]+\frac{\beta_1}{h} [u_h]+\partial_\nu q_h,
\ea
\eq
with jump terms defined as
\bq
\ba\
 [u_h]= g_1 - u_h,  \quad \text{or } & [u_h]= 0;\\
[q_h]= -g_3 - q_h, \quad \text{or } & [q_h]= 0.
\ea
\eq
\fi
for case (i), (ii) and (iii), respectively:
 \begin{subequations}\label{fluxBD1}
 \begin{align}
 & \widehat{u_h} = g_1, \;
\widehat{\partial_\nu u_h} =  g_2, \\
&  \widehat{q_h} =  q_h, \;
 \widehat{\partial_\nu q} = \frac{\beta_1}{h} (g_1-u_h)+ \partial_\nu q_h;
\end{align}
\end{subequations}
 \begin{subequations}\label{fluxBD2}
  \begin{align}
& \widehat{u_h} = g_1, \;
\widehat{\partial_\nu u_h} =\frac{\beta_0}{h} (g_1-u_h)+ \partial_\nu u_h, \\
&  \widehat{q_h} =  -g_3, \;
 \widehat{\partial_\nu q} = \frac{\beta_0}{h} (-g_3-q_h) + \partial_\nu q_h;
\end{align}
\end{subequations}
\begin{subequations}\label{fluxBD3}
 \begin{align}
& \widehat{u_h} = u_h, \;
\widehat{\partial_\nu u_h} =  g_2, \\
 & \widehat{q_h} =  q_h, \;
 \widehat{\partial_\nu q} =-g_4,
 \end{align}
 \end{subequations}
where the mesh size $h=diam\{K\}$.  %boundary data are used whenever available, else we use the boundary trace of the corresponding numerical values.
The flux parameters $\beta_{0}, \ \beta_{1}$ are used to ensure the numerical convergence.
%generally there is a $w$, such that
%$w$ and $u$ satisfy the same boundary condition and
%$$
%(u-w)_t=-\Delta^2 (u-w)+f, \quad u(x, 0)-w(x,0)=u_0(x)-w(x,0), x\in \Omega, \quad t>0,
%$$
%where $f=-(w_t+\Delta^2w)$ and the equation is subject to zero boundary conditions \cite{LM72}.
%
%Thus, we will consider the stability of the DG scheme corresponding to equations with zero boundary conditions and a source term $f$, and
For these three types of boundary fluxes, the following stability results hold true.
\begin{thm}
The DG scheme (\ref{FPDG1D}) or (\ref{FPDG1D+}) subject to one of three types of boundary fluxes (\ref{fluxBD1})-(\ref{fluxBD3}) is stable in the sense that
\bq\label{stabbd}
\|u_h - \widetilde{u}_h\| \leq \|u_0-\widetilde{u}_0\|,
\eq
provided (i)   $\beta_1 \geq 0$,   (ii) $\forall \beta_0$,  and (iii) no flux parameter is needed.  Here $u_h$ and $\widetilde{u}_h$ in (\ref{stabbd}) denote the corresponding numerical solutions that satisfy the same boundary conditions associated with the initial conditions $u_0$ and $\widetilde{u}_0$, respectively.
\end{thm}
\begin{proof}
Let $A^0(\cdot, \cdot)$ be the bilinear operator defined in (\ref{FPDG1D}) or (\ref{FPDG1D+}), yet without boundary terms. Then the sum of two global formulations
yields the following
\bq\label{semiDGCom}
(u_{ht}, \phi) +(q_h, \psi)=-A^0(q_h, \phi)+A^0(u_h, \psi)+B(u_h, q_h; \phi, \psi),
\eq
where
$$
B=\int_\Gamma \left( \widehat{\partial_\nu q_h} \phi - \widehat{\partial_\nu u_h}\psi +(q_h-\widehat q_h)\partial_\nu \phi
-(u_h-\widehat u_h)\partial_\nu \psi \right)ds.
$$
%Here $\Gamma$ is a union of all boundary edges in 2D case, and $\{x_{1/2}=a, b=x_{N+1/2}\}$ in one-dimensional case.

Upon careful calculation,  $B$ in each case is given as follows:\\
(i)
$$
B=\int_{\Gamma} (\partial_\nu q_h \phi-u_h \partial_\nu \psi -\frac{\beta_1}{h} u_h \phi )ds +\int_\Gamma (g_1\partial_\nu \psi  - g_2\psi+\frac{\beta_1}{h}  g_1\phi)ds;
$$
(ii)
\begin{align*}
B& =\int_{\Gamma} \left(\partial_\nu q_h \phi+ q_h \partial_\nu \phi -u_h \partial_\nu \psi   -\partial_{\nu} u_h \psi \right)ds  -\int_\Gamma\left( \frac{\beta_0}{h}(q_h\phi -u_h\psi)\right)ds \\
& \qquad +\int_\Gamma (\frac{\beta_0}{h} (-g_3\phi -g_1\psi)+g_1\partial_\nu \psi +g_3\partial_\nu \phi)ds; \text{and}
\end{align*}
 (iii)
$$
B=\int_{\Gamma} (-g_4\phi -g_2\psi)ds.
$$
Taking $\phi=u_h$ and $\psi=q_h$ in  (\ref{semiDGCom}) %$$
%(u_{ht}, \phi) +(q_h, \psi)=-A^0(q_h, \phi)+A^0(u_h, \psi)+B(u_h, q_h; \phi, \psi),
%$$
we obtain
$$
\frac{1}{2} \frac{d}{dt} \|u_{h}\|^2+\|q_h\|^2 =B(u_h, q_h; u_h, q_h),
$$
where such $B$  reduces to  \\
(i)
$
B=-\frac{\beta_1}{h}  \int_{\Gamma}  u_h^2 ds +\int_\Gamma (g_1\partial_\nu q_h  - g_2q_h+\frac{\beta_1}{h}  g_1u_h)ds,
$ \\
(ii) $B =\int_\Gamma (\frac{\beta_0}{h} (-g_3u_h -g_1q_h)+g_1\partial_\nu q_h +g_3\partial_\nu u_h)ds$, and \\
(iii) $ B=\int_{\Gamma} (-g_4u_h -g_2q_h)ds.$\\
Both the equation and the boundary conditions are linear, it suffices to show $\|u_h(\cdot, t)\|\leq \|u_h(\cdot, 0)\|$ when boundary conditions are homogeneous, i.e., $g_i=0, \ i=1, \cdots, 4$.  Indeed, in such cases we have  (i) $B=-\frac{\beta_1}{h}\int_{\Gamma} u_h^2ds$,  (ii) $B=0\quad  \forall \beta_0$, and   (iii) $B=0$.
\iffalse
{\color{red}
Since $u_h$ and $\widetilde{u}_h$ are solutions of (\ref{semiDGCom}) associated with the initial conditions $u_0$ and $\widetilde{u}_0$, then it follows
\bq\label{semiDGComDF}
(\widetilde{\widetilde{u}}_{ht}, \phi) +(\widetilde{\widetilde{q}}_h, \psi)=-A^0(\widetilde{\widetilde{q}}_h, \phi)+A^0(\widetilde{\widetilde{u}}_h, \psi)+\widetilde{B}(\widetilde{\widetilde{u}}_h, \widetilde{\widetilde{q}}_h; \phi, \psi),
\eq
where $\widetilde{\widetilde{u}}_h=u_h-\widetilde{u}_h$,  $\widetilde{\widetilde{q}}_h=q_h-\widetilde{q}_h$ and $\widetilde{B}(\widetilde{\widetilde{u}}_h, \widetilde{\widetilde{q}}_h; \phi, \psi)$ is obtained by taking $g_i=0, \ i=1, \cdots, 4$ in $B(\widetilde{\widetilde{u}}_h, \widetilde{\widetilde{q}}_h; \phi, \psi)$.

With the three conditions (i)-(iii) for $\beta_0, \ \beta_1$ and further taking $\phi=\widetilde{\widetilde{u}}_h, \; \psi=\widetilde{\widetilde{q}}_h$ in (\ref{semiDGComDF}), we have
$$
\widetilde{B}(\widetilde{\widetilde{u}}_h, \widetilde{\widetilde{q}}_h; \widetilde{\widetilde{u}}_h, \widetilde{\widetilde{q}}_h) \leq 0,
$$
which implies
\bqs
\frac{1}{2}\frac{d}{dt} \|\widetilde{\widetilde{u}}_h\|^2 +\| \widetilde{\widetilde{q}}_h \|^2\leq 0.
\eqs
}
\fi
Thus, the conclusion follows.
\end{proof}

\begin{rem} If $\tilde u_h$ is an approximation to the steady solution of the corresponding time-independent problem, then (\ref{stabbd}) leads to
$$
\|u_h \| \leq \|u_0-\widetilde{u}_0\|+\|\widetilde{u}_h\|,
$$
which can be regarded as the priori bound in terms of both initial data and the boundary data.
\end{rem}

The necessity of using $\beta_1$ in  (\ref{fluxBD1}) and $\beta_0$ in (\ref{fluxBD2}) is illustrated numerically in Example \ref{Ex1dbd1} and \ref{Ex1dbd2}, respectively,  by checking whether the optimal order of accuracy can be obtained. Extensive numerical tests including Example \ref{Ex1dbd1} and \ref{Ex1dbd2} indicate that the choice of $\beta_0, \beta_1$ as shown in Table \ref{tabbeta} is sufficient for achieving optimal convergence.
\begin{table}[!htbp]\tabcolsep0.03in
\caption{The choice of $\beta_0, \beta_1$  in boundary fluxes (\ref{fluxBD1}) and (\ref{fluxBD2}).}
\begin{tabular}[c]{||c|c|c||}
\hline
% \multirow{2}{*}{$k$} & \multirow{2}{*}{$\Delta t$}&   \multirow{2}{*}{ } \\
fluxes & $k=1$ & $k\geq 2$ \\
\hline
(\ref{fluxBD1}) & $\beta_1=0$ & $\beta_1 \geq \delta $ \\
\hline
(\ref{fluxBD2})  & $|\beta_0|\geq C$ & $\beta_0=0$ \\
\hline
\end{tabular}\label{tabbeta}
\end{table}
In  Table 1, $\delta>0$ can be a quite small number (see Figure \ref{1dbeta1}), and $C>0$ is a constant, say $C=3$ is a valid choice in our numerical examples {on uniform meshes}. It would be interesting to justify  these sufficient conditions by establishing some optimal error estimates.

\subsection{Application to a nonlinear problem} We consider the initial-boundary value problem for the  one-dimensional Swift-Hohenberg equation  of the form,
\begin{equation}\label{SH}
\left \{
\begin{array}{rl}
% & u_t=-D(\partial_x^2+\kappa^2)^2u+f(u), \\
 & u_t=-D\kappa^4 u-2D\kappa^2 u_{xx} -D u_{xxxx}+f(u) \quad x \in [a,b], \; t>0,  \\
& u=0 \; \text{and} \; u_{xx}=0 \quad \text{at} \; x=a, b,\\
% & u(a,t)= 0, \quad u(b,t)=0, \\
 %& u_{xx}(a,t)= 0, \quad u_{xx}(b,t)=0,\\
 & u(x,0)=u_0(x),
\end{array}
\right.
\end{equation}
where $D>0, \ \kappa$ are constants and $f(u)= \varepsilon u+gu^2-u^3$
%$
%f(u)=-\frac{d\Phi(u)}{du} = \varepsilon u+gu^2-u^3, \quad \Phi(u) = -\frac{\varepsilon}{2}u^2-\frac{g}{3}u^3+\frac{1}{4}u^4,
%$
with non-negative constants $\varepsilon, \ g$.  The Swift-Hohenberg equation introduced in \cite{SH77} is noted for its pattern-forming behavior, and
 endowed with a gradient flow structure, $u_t= -\frac{\delta \mathcal{E}}{\delta u}$, for zero-flux boundary conditions.  This equation relates the temporal evolution of the pattern to the spatial structure of the pattern, with $\epsilon$ measuring how far the temperature is above the minimum temperature difference required for convection, and  $g$ is the parameter controlling the strength of the quadratic nonlinearity.

The Swift-Hohenberg equation (\ref{SH}) can be rewritten as an equivalent system
\begin{equation*}%\label{BeamSys1DNon}
\left \{
\begin{array}{rl}
    u_t = & \sqrt{D}(\partial_x^2+\kappa^2)q+f(u),\\
    q = & -\sqrt{D}(\partial_x^2+\kappa^2)u.
\end{array}
\right.
\end{equation*}
With this formulation the energy dissipation law becomes
%One can check that (\ref{BeamSys1DNon}) satisfies an energy dissipation law of the form
\bqs
\frac{d}{dt}\mathcal{E} = -\int_a^b|u_t|^2dx \leq 0,
\eqs
where
$
\mathcal{E}=\int_a^b \Phi(u)+\frac{1}{2} |q|^2dx
$
is a free-energy functional,  and
$$
\Phi(u) = -\frac{\varepsilon}{2}u^2-\frac{g}{3}u^3+\frac{1}{4}u^4= -\int_0^u f(\xi)d\xi.
$$
This is the fundamental stability property of the Swift--Hohenberg equation. The objective of this section is to illustrate that our DG discretization with proper time discretization inherits this property irrespectively  of time step sizes.

The semi-discrete DG method for (\ref{SH}) may be given by
%\bqs%\label{FPDG1DNon}
\begin{align*}
(u_{ht}, \phi) = &- A(q_h,\phi) +(f(u_h), \phi),\\
(q_h, \psi) = & A(u_h,\psi),
\end{align*}
%\eqs
where
\begin{align*}
A(w,v)= & \sqrt{D} \left( A^0(w,v)-\kappa^2(w, v)+\left( w^+v^+_x\right)_{1/2}-\left( w^-v^-_x\right)_{N+1/2} \right. \\
&\left. +\left( w_x^+v^+\right)_{1/2}-\left(  w_x^-v^-\right)_{N+1/2}+\frac{\beta_{0}}{h}\left( w^+{v^+}\right)_{1/2}
+\frac{\beta_{0}}{h}\left( w^-{v^-}\right)_{N+1/2} \right ),
\end{align*}
with parameter $\beta_0$ chosen as listed in Table \ref{tabbeta}. The DG scheme can be shown to preserve the energy dissipation law in the sense that
\bqs
\frac{d}{dt}\mathcal{E}_h  = -\int_a^b|u_{ht}|^2dx \leq 0,
\eqs
where
$
\mathcal{E}_h=\int_a^b \Phi(u_h)+\frac{1}{2} |q_h|^2dx.
$

Time discretization should be taken with care, here we want to preserve the energy dissipation law at each time step.  A simple choice is
%Here a fully discrete DG scheme for (\ref{BeamSys1DNon}) is
to obtain $(u_h^{n+1}, q_h^{n+1}) \in V_h^k \times V_h^k$ from $(u_h^{n}, q_h^{n})$ by
\begin{subequations}\label{FPDGFullNon}
\begin{align}
   \left(  \frac{u_h^{n+1} - u_h^n}{\Delta t}, \phi \right) = & - A(q_h^{n+1/2},\phi)- \left( \frac{\Phi(u_h^{n+1})-\Phi(u^n)}{u_h^{n+1}-u_h^n},\phi\right)\\
    (q_h^{n}, \psi) = & A(u_h^{n},\psi),
\end{align}
\end{subequations}
for all $\phi, \ \psi \in V_h^k$, where $q_h^{n+1/2}=\frac{1}{2}(q^{n+1}_h+q_h^n)$. This fully discrete DG scheme does have the following property.
%For the fully discrete DG scheme (\ref{FPDGFullNon}), we have
\begin{thm}
The solution to (\ref{FPDGFullNon}) satisfies the energy dissipation law of the form
\bq\label{engdis}
\mathcal{E}_h^{n+1} - \mathcal{E}_h^n = -\frac{\|u_h^{n+1}-u_h^{n}\|^2}{\Delta t} ,
\eq
where
$$
\mathcal{E}_h^n= \int_a^b \Phi(u_h^n) +\frac{1}{2}|q_h^n|^2dx.
$$
\end{thm}
\begin{proof}
By taking the difference of (\ref{FPDGFullNon}b) at time level $n+1$ and $n$, we obtain
\bq\label{FPDGFull1DDe}
(q_h^{n+1}-q_h^{n}, \psi) = A(u_h^{n+1}-u_h^{n},\psi).
\eq
Taking $\phi=u_h^{n+1} - u_h^n$ in (\ref{FPDGFullNon}a), $\psi=q_h^{n+1/2 }$ in (\ref{FPDGFull1DDe}), then
plugging the resulting relation into (\ref{FPDGFullNon}a), we have
\bqs\ba
\frac{\|u_h^{n+1}-u_h^{n}\|^2}{\Delta t} = & - \left( q_h^{n+1}-q_h^{n}, q_h^{n+1/2}\right)- \int_a^b (\Phi(u_h^{n+1})- \Phi(u_h^{n})) dx\\
= &  - \frac{1}{2}\left(\|q_h^{n+1}\|^2-\|q_h^{n}\|^2\right) -\int_a^b \Phi(u_h^{n+1})dx +\int_a^b \Phi(u_h^{n}) dx,\\
\ea\eqs
which gives (\ref{engdis}).
\end{proof}
We next propose an iteration scheme to solve the nonlinear equation (\ref{FPDGFullNon}).  Rewriting the nonlinear term in (\ref{FPDGFullNon}a) as
$$
\frac{\Phi(u_h^{n+1})-\Phi(u_h^n)}{u_h^{n+1}-u_h^n}=G_1(u_h^{n+1},u_h^n)u_h^{n+1}+G_2(u_h^n),
$$
where
\bqs
\ba
G_1(w,v)=& -\frac{\varepsilon}{2}-\frac{g}{3}(w+v)+\frac{1}{4}(w^2+wv+v^2), \\
G_2(v)=& -\frac{\varepsilon}{2}v-\frac{g}{3}v^2+\frac{1}{4}v^3,
\ea
\eqs

with which we apply  the idea in \cite{CP02} to obtain the following iterative scheme,
\bq\label{FPDGFull1DNon}
\ba
     \left( \frac{u_h^{n+1,l+1}-u_h^n}{\Delta t}, \phi \right) & +
     \frac{1}{2}A(q_h^{n+1,l+1},\phi)=
    -\frac{1}{2}A(q_h^{n},\phi)-\left( G_1(u_h^{n+1,l},u_h^n)u_h^{n+1,l+1}+G_2(u_h^n), \phi \right),\\
   \frac{1}{2}A(u_h^{n+1,l+1},\psi) & -\frac{1}{2}(q_h^{n+1,l+1}, \psi) = 0,% -\frac{1}{2}A(u_h^{n},\psi)+\frac{1}{2}(q_h^{n}, \psi),
\ea
\eq
where $G_1(u_h^{n+1,0},u_h^n)=G_1(u_h^n,u_h^n)$, the iteration stops as $\|u_h^{n+1,l}-u_h^{n+1,l-1}\| < \delta$ for certain $l=L \ (L \geq 1)$.
Then the last iteration gives the sought solution on the new time stage and we define
\bqs%\label{IterSoln}
u_h^{n+1} =u_h^{n+1,L}.
\eqs
%\begin{rem}
%In this section, we only considered one-dimensional Swift-Hohenberg equation, we will explore more general nonlinear 4th order PDEs in future.
%\end{rem}

\section{Numerical examples}
In this section we numerically validate our theoretical results, as well as the stated extension cases with one and two dimensional examples.
%are presented to investigate
%The $L^2$ error on each computational cells is measured by Gaussian quadrature rule, and the $L^\infty$ error is taken as the maximum absolute value of the difference of numerical solution with the exact solution.
The orders of convergence are calculated using
$$
\log_2\frac{\|u-u_h\|_{L^2}}{\|u-u_{h/2}\|_{L^2}}, \quad \log_2\frac{\|u-u_h\|_{L^\infty}}{\|u-u_{h/2}\|_{L^\infty}},
$$
respectively; where errors are given in the following way: for the one dimensional $L^2$ error, we use
$$
\|u-u_h\|_{L^2} = \left( \sum_{j=1}^N \frac{h_j}{2} \sum_{\alpha=1}^{k+1} \omega_{\alpha} |u_h(\hat{x}^j_{\alpha},t)-u(\hat{x}^j_{\alpha},t)|^2 \right)^{\frac{1}{2}},
$$
where $\omega_{\alpha}>0$ are the weights, $\hat{x}^j_{\alpha}$ are the corresponding Gauss points in each cell $I_j$,  and
for the $L^\infty$ error,
$$
\|u-u_h\|_{L^\infty} = \max_{1\leq j \leq N}  \max_{1\leq\alpha\leq k+1} |u_h(\hat{x}^j_{\alpha},t)-u(\hat{x}^j_{\alpha},t)|.
$$

\begin{example}\label{Ex1d} (1D accuracy test)
We consider the biharmonic equation
\begin{equation*}%\label{Ex1Biharmonic}
\left \{
\begin{array}{rl}
    u_t &=  -u_{xxxx} \quad (x,t) \in [0, 2\pi]\times (0,T],\\
    u(x,0) &=   \sin(x),\\
\end{array}
\right.
\end{equation*}
with periodic boundary conditions.  And the exact solution is given by
\bqs
    u(x,t) = e^{-t} \sin(x).
\eqs
We test this example using DG scheme (\ref{FPDGCell}) with the Crank-Nicolson time discretization, based on polynomials of degree $k$ with $k=1, \cdots, 4$. Both errors and orders of accuracy at $T=1$ are reported in Table \ref{tab1dtk1}. These results show that $(k+1)$th order of accuracy in both $L^2$ and $L^{\infty}$ are obtained.

\begin{table}[!htbp]\tabcolsep0.03in
\caption{1D $L^2, \; L^\infty$ errors for biharmonic equation at $T= 1$.}
\begin{tabular}[c]{||c|c|c|c|c|c|c|c|c|c||}
\hline
\multirow{2}{*}{$k$} & \multirow{2}{*}{$\Delta t$}&   \multirow{2}{*}{ } & N=10 & \multicolumn{2}{|c|}{N=20} & \multicolumn{2}{|c|}{N=40} & \multicolumn{2}{|c||}{N=80}  \\
\cline{4-10}
& & & error & error & order & error & order & error & order\\
\hline
\multirow{2}{*}{1}  & \multirow{2}{*}{0.01} & $\|u-u_h\|_{L^2}$ &  0.0507931 & 0.0113953 & 2.16  & 0.00278271 & 2.03 & 0.000694474 & 2.00 \\
%\cline{3-10}
% & & $|u-u_h|_{H^1}$  & 0.118437 & 0.0591038 &  1.00 & 0.0304443 & 0.96  & 0.0149446 & 1.03 \\
\cline{3-10}
 & & $\|u-u_h\|_{L^\infty}$  & 0.0341444 & 0.00769913 & 2.15  & 0.00189324 &  2.03 & 0.000475639 & 1.99 \\
\hline
\hline
\multirow{2}{*}{2}  & \multirow{2}{*}{0.0005} & $\|u-u_h\|_{L^2}$ & 0.00395192 & 0.000559636 & 2.82  & 7.24864e-05 & 2.95 & 8.7753e-06 & 3.05 \\
%\cline{3-10}
% & & $|u-u_h|_{H^1}$  & 0.021178  & 0.00571154 &  1.89 & 0.00145912 & 1.97 & 0.000352647 & 2.05 \\
\cline{3-10}
 & & $\|u-u_h\|_{L^\infty}$  & 0.00296885 & 0.000444451 & 2.74 & 5.84061e-05 &  2.93 & 7.0346e-06 & 3.05 \\
\hline
\hline
\multirow{2}{*}{$k$} & \multirow{2}{*}{$\Delta t$}&   \multirow{2}{*}{ } & N=5 & \multicolumn{2}{|c|}{N=10} & \multicolumn{2}{|c|}{N=20} & \multicolumn{2}{|c||}{N=40}  \\
\cline{4-10}
& & & error & error & order & error & order & error & order\\
\hline
\multirow{2}{*}{3} & \multirow{2}{*}{0.0005}  & $\|u-u_h\|_{L^2}$ & 0.000716136 & 3.6469e-05 &  4.30  & 2.14439e-06 & 4.09 & 1.18333e-07 & 4.18 \\
%\cline{3-10}
% & & $|u-u_h|_{H^1}$   & 0.00467751  & 0.000528519  & 3.15 & 0.000528519 &  3.04 & 8.12922e-06 & 2.98 \\
\cline{3-10}
 & & $\|u-u_h\|_{L^\infty}$   & 0.000580818  & 3.17668e-05 & 4.19 & 1.87677e-06 & 4.08  & 1.1109e-07 & 4.08 \\
\hline
\hline
\multirow{2}{*}{4} & \multirow{2}{*}{0.0001} &  $\|u-u_h\|_{L^2}$ & 5.25422e-05 & 1.95246e-06 & 4.75 &  6.42678e-08 & 4.93 & 2.07446e-09 & 4.95 \\
%\cline{3-10}
% & & $|u-u_h|_{H^1}$   &  0.000434364 & 3.04635e-05 & 3.83 & 1.96911e-06 & 3.95 & 1.2179e-07 & 4.02 \\
\cline{3-10}
 & & $\|u-u_h\|_{L^\infty}$   &  3.98997e-05 & 1.60107e-06 & 4.64 & 5.42808e-08 & 4.88 & 1.69245e-09 & 5.00 \\
\hline
\end{tabular}\label{tab1dtk1}
\end{table}

\end{example}

\begin{example}\label{Ex2d} (2D accuracy test)
We consider the 2D linear biharmonic equation
\begin{equation*}%\label{Ex1Biharmonic2D}
\left \{
\begin{array}{rl}
    u_t + \Delta^2 u = & 0 \quad (x,y,t) \in [0, 4\pi]\times [0, 4\pi]\times (0,T],\\
    u(x,y,0) = &  \sin(0.5x)\sin(0.5y),\\
\end{array}
\right.
\end{equation*}
with periodic boundary conditions.  And the exact solution is given by
\bqs
    u(x,t) = e^{-0.25t} \sin(0.5x)\sin(0.5y).
\eqs
We test this example by DG scheme (\ref{FPDGFull+}) with $\theta=1/2$,  based on tensor product of polynomials of degree $k$ with $k=1, 2, 3$ on rectangular meshes.  Both errors and orders of accuracy at $T=0.1$ are reported in Table \ref{tab2dtk1}. These results show that $(k+1)$th order of accuracy in both $L^2$ and $L^{\infty}$ are obtained.
\begin{table}[!htbp]\tabcolsep0.03in
\caption{2D $L^2, \; L^\infty$ errors for biharmonic equation at $T= 0.1$.}
\begin{tabular}[c]{||c|c|c|c|c|c|c|c|c|c||}
\hline
\multirow{2}{*}{$k$} & \multirow{2}{*}{$\Delta t$}&   \multirow{2}{*}{ } & N=8 & \multicolumn{2}{|c|}{N=16} & \multicolumn{2}{|c|}{N=32} & \multicolumn{2}{|c||}{N=64}  \\
\cline{4-10}
& & & error & error & order & error & order & error & order\\
%\hline
%\multirow{3}{*}{0}  & \multirow{3}{*}{1e-3} & $\|u-u_h\|_{L^2}$ &  1.9258 & 0.988716 & 0.96 & 0.514309 & 0.94 & 0.290359 & 0.82  \\
%\cline{3-10}
% & & $|u-u_h|_{H^1}$  & 4.33319 & 4.33319 & 0.00 & 4.33319 & 0.00 & 4.33319 & 0.00  \\
%\cline{3-10}
% & & $\|u-u_h\|_{L^\infty}$  & 0.353899 & 0.194848 & 0.86 & 0.107973 & 0.85 & 0.0618326 & 0.80  \\
%\hline
\hline
\multirow{2}{*}{1}  & \multirow{2}{*}{1e-3} & $\|u-u_h\|_{L^2}$ &  0.294331 & 0.0617401 & 2.25 & 0.0132547 & 2.22 & 0.00316944 & 2.06  \\
%\cline{3-10}
% & & $|u-u_h|_{H^1}$  & 1.09765 & 0.51869 & 1.08 & 0.248985 & 1.06 & 0.123212 & 1.01  \\
\cline{3-10}
 & & $\|u-u_h\|_{L^\infty}$  & 0.113491 & 0.0259853 & 2.13 & 0.00620769 & 2.07 & 0.0015334 & 2.02  \\
\hline
\hline
\multirow{2}{*}{2}  & \multirow{2}{*}{1e-4} & $\|u-u_h\|_{L^2}$ & 0.0857554 & 0.0138187 & 2.63 & 0.00185713 & 2.90 & 0.000232547 & 3.00  \\
%\cline{3-10}
% & & $|u-u_h|_{H^1}$  & 0.197434 & 0.057513 & 1.78 & 0.0150231 & 1.94 & 0.00373757 & 2.01  \\
\cline{3-10}
 & & $\|u-u_h\|_{L^\infty}$  & 0.015608 & 0.00239088 & 2.71 & 0.000311659 & 2.94 & 3.86222e-05 & 3.01  \\
 \hline
\hline
\multirow{2}{*}{$k$} & \multirow{2}{*}{$\Delta t$}&   \multirow{2}{*}{ } & N=4 & \multicolumn{2}{|c|}{N=8} & \multicolumn{2}{|c|}{N=16} & \multicolumn{2}{|c||}{N=32}  \\
\cline{4-10}
& & & error & error & order & error & order & error & order\\
\hline
\multirow{2}{*}{3} & \multirow{2}{*}{1e-5}  & $\|u-u_h\|_{L^2}$ & 0.0241859 & 0.00123277 & 4.29 & 7.05843e-05 & 4.13 & 4.31039e-06 & 4.03  \\
%\cline{3-10}
% & & $|u-u_h|_{H^1}$   & 0.0626366 & 0.00702221 & 3.16 & 0.000839367 & 3.06 & 0.000103717 & 3.02  \\
\cline{3-10}
 & & $\|u-u_h\|_{L^\infty}$   & 0.00353992 & 0.000355156 & 3.32 & 2.00749e-05 & 4.14 & 1.50258e-06 & 3.74  \\
\hline
\end{tabular}\label{tab2dtk1}
\end{table}

\end{example}

\begin{example}\label{Ex2duxx} (2D linearized Cahn-Hillard equation)
We consider the 2D linearized Cahn-Hillard equation
\begin{align*}%\label{Ex1Biharmonic2Duxx}
    & u_t + \Delta^2 u + \Delta u =  0 \quad (x,y,t) \in [0, 2\pi/a]\times [0, 2\pi/a]\times (0,T],\\
    & u(x,y,0) =  \sin(ax)\sin(ay),
\end{align*}
with periodic boundary conditions, where $a>0$ is a constant.

The exact solution is given by
\bqs
    u(x,t) = e^{-bt} \sin(ax)\sin(ay),
\eqs
where $b=4a^4-2a^2$.
  % with the generalized Neumann condition was studied in \cite{DS09} by the local discontinuous Galerkin method.

We test this example using DG scheme  (\ref{FPDG1D++Ex})  on rectangular meshes with the Crank-Nicolson time discretization, based on polynomials of degree $k$ with $k=1, 2, 3$, by varying the interval length through $a$ in three cases:
(i) $a=1/2$; (ii) $a=\sqrt{2}/2$; and (iii) $a=\sqrt{3}/2$. They correspond to $b=-1/4, 0, 3/4$,  while the solution in each case shows different growth/decay behavior in time.

%\noindent \textbf{Case 1.} , then $b=-1/4$.
%\noindent \textbf{Case 2.}  $a=\sqrt{2}/2$, then $b=0$.
%\noindent \textbf{Case 3.}  $a=\sqrt{3}/2$, then $b=3/4$.

Both errors and orders of accuracy at $T=0.1$ are reported in Table \ref{tab2dtk1uxx}-\ref{tab2dtk1uxx3}, respectively. These results show that $(k+1)$th order of accuracy in both $L^2$ and $L^{\infty}$ norms are obtained.

%For the Test cases above, the penalty free DG method based on polynomials of degree $k$ with $k=1, 2, 3$ on rectangular mesh are tested, $(k+1)$th order of accuracy for $L^2$, $k$th order of accuracy for $H^1$ and $(k+1)$th order of accuracy for $L^{\infty}$ are obtained at $t=0.1$ for $k=1, 2, 3$. The errors and accuracy are reported in Table \ref{tab2dtk1uxx}-\ref{tab2dtk1uxx3}, respectively.

%The DG method based on polynomials of degree $m$ with $m=0, 1, 2, 3$ on rectangular mesh are tested, $(m+1)$th order of accuracy for $L^2$, $m$th order of accuracy for $H^1$ and $(m+1)$th order of accuracy for $L^{\infty}$ are obtained at $t=0.1$ and the errors and accuracy are reported in Table \ref{tab2dtk1uxx}. Here we notice that the penalty free DG scheme can get optimal errors for $P^0$ polynomial.

\begin{table}[!htbp]\tabcolsep0.03in
\caption{2D $L^2, \; L^\infty$ errors for linearized Cahn-Hillard equation at $T= 0.1$, a=1/2.}
\begin{tabular}[c]{||c|c|c|c|c|c|c|c|c|c||}
\hline
\multirow{2}{*}{$k$} & \multirow{2}{*}{$\Delta t$}&   \multirow{2}{*}{ } & N=8 & \multicolumn{2}{|c|}{N=16} & \multicolumn{2}{|c|}{N=32} & \multicolumn{2}{|c||}{N=64}  \\
\cline{4-10}
& & & error & error & order & error & order & error & order\\
%\hline
%\multirow{3}{*}{0}  & \multirow{3}{*}{1e-3} & $\|u-u_h\|_{L^2}$ &  2.02425 & 1.0388 & 0.96 & 0.539485 & 0.95 & 0.303117 & 0.83  \\
%\cline{3-10}
% & & $|u-u_h|_{H^1}$  & 4.55536 & 4.55536 & 0.00 & 4.55536 & 0.00 & 4.55536 & 0.00  \\
%\cline{3-10}
% & & $\|u-u_h\|_{L^\infty}$  & 0.389258 & 0.213786 & 0.86 & 0.114002 & 0.91 & 0.0643459 & 0.83  \\
%\hline
\hline
\multirow{2}{*}{1}  & \multirow{2}{*}{1e-3} & $\|u-u_h\|_{L^2}$ &  0.334674 & 0.0647558 & 2.37 & 0.0138946 & 2.22 & 0.00332186 & 2.06  \\
%\cline{3-10}
% & & $|u-u_h|_{H^1}$  & 1.19824 & 0.546324 & 1.13 & 0.261876 & 1.06 & 0.129545 & 1.02  \\
\cline{3-10}
 & & $\|u-u_h\|_{L^\infty}$  & 0.126283 & 0.0280333 & 2.17 & 0.00669205 & 2.07 & 0.00165341 & 2.02  \\
\hline
\hline
\multirow{2}{*}{2}  & \multirow{2}{*}{1e-4} & $\|u-u_h\|_{L^2}$ & 0.090608 & 0.0145271 & 2.64 & 0.00195239 & 2.90 & 0.000248728 & 2.97  \\
%\cline{3-10}
% & & $|u-u_h|_{H^1}$  & 0.208267 & 0.0604266 & 1.79 & 0.0157846 & 1.94 & 0.00399543 & 1.98  \\
\cline{3-10}
 & & $\|u-u_h\|_{L^\infty}$  & 0.0165817 & 0.00251807 & 2.72 & 0.00032726 & 2.94 & 4.12504e-05 & 2.99  \\
 \hline
\hline
\multirow{2}{*}{$k$} & \multirow{2}{*}{$\Delta t$}&   \multirow{2}{*}{ } & N=4 & \multicolumn{2}{|c|}{N=8} & \multicolumn{2}{|c|}{N=16} & \multicolumn{2}{|c||}{N=32}  \\
\cline{4-10}
& & & error & error & order & error & order & error & order\\
\hline
\multirow{2}{*}{3} & \multirow{2}{*}{1e-5}  & $\|u-u_h\|_{L^2}$ & 0.0250808 & 0.00129598 & 4.27 & 7.42033e-05 & 4.13 & 4.53139e-06 & 4.03  \\
%\cline{3-10}
% & & $|u-u_h|_{H^1}$   & 0.0651707 & 0.00738216 & 3.14 & 0.000882402 & 3.06 & 0.000109035 & 3.02  \\
\cline{3-10}
 & & $\|u-u_h\|_{L^\infty}$   & 0.00365516 & 0.000373252 & 3.29 & 2.48922e-05 & 3.91 & 1.57959e-06 & 3.98  \\
\hline
\end{tabular}\label{tab2dtk1uxx}
\end{table}

\begin{table}[!htbp]\tabcolsep0.03in
\caption{2D $L^2, \; L^\infty$ errors for linearized Cahn-Hillard equation at $T = 0.1$, $a=\sqrt{2}/2$.}
\begin{tabular}[c]{||c|c|c|c|c|c|c|c|c|c||}
\hline
\multirow{2}{*}{$k$} & \multirow{2}{*}{$\Delta t$}&   \multirow{2}{*}{ } & N=8 & \multicolumn{2}{|c|}{N=16} & \multicolumn{2}{|c|}{N=32} & \multicolumn{2}{|c||}{N=64}  \\
\cline{4-10}
& & & error & error & order & error & order & error & order\\
%\hline
%\multirow{3}{*}{0}  & \multirow{3}{*}{1e-3} & $\|u-u_h\|_{L^2}$ &  1.39213 & 0.708175 & 0.98 & 0.355624 & 0.99 & 0.178005 & 1.00  \\
%\cline{3-10}
% & & $|u-u_h|_{H^1}$  & 4.44288 & 4.44288 & 0.00 & 4.44288 & 0.00 & 4.44288 & 0.00  \\
%\cline{3-10}
% & & $\|u-u_h\|_{L^\infty}$  & 0.371358 & 0.197248 & 0.91 & 0.0991187 & 0.99 & 0.0494081 & 1.00  \\
%\hline
\hline
\multirow{2}{*}{1}  & \multirow{2}{*}{1e-3} & $\|u-u_h\|_{L^2}$ &  0.271457 & 0.0450757 & 2.59 & 0.00969181 & 2.22 & 0.00229956 & 2.08  \\
%\cline{3-10}
% & & $|u-u_h|_{H^1}$  &  1.23197 & 0.530803 & 1.21 & 0.255168 & 1.06 & 0.126089 & 1.02  \\
\cline{3-10}
 & & $\|u-u_h\|_{L^\infty}$  &  0.122082 & 0.0259627 & 2.23 & 0.00620589 & 2.06 & 0.00152936 & 2.02  \\
\hline
\hline
\multirow{2}{*}{2}  & \multirow{2}{*}{1e-4} & $\|u-u_h\|_{L^2}$ & 0.0627901 & 0.0100189 & 2.65 & 0.00134647 & 2.90 & 0.000171541 & 2.97  \\
%\cline{3-10}
% & & $|u-u_h|_{H^1}$  & 0.20484 & 0.0590032 & 1.80 & 0.0153997 & 1.94 & 0.00389411 & 1.98  \\
\cline{3-10}
 & & $\|u-u_h\|_{L^\infty}$  & 0.0161613 & 0.0024469 & 2.72 & 0.000318576 & 2.94 & 3.99023e-05 & 3.00  \\
 \hline
\hline
\multirow{2}{*}{$k$} & \multirow{2}{*}{$\Delta t$}&   \multirow{2}{*}{ } & N=4 & \multicolumn{2}{|c|}{N=8} & \multicolumn{2}{|c|}{N=16} & \multicolumn{2}{|c||}{N=32}  \\
\cline{4-10}
& & & error & error & order & error & order & error & order\\
\hline
\multirow{2}{*}{3} & \multirow{2}{*}{1e-5}  & $\|u-u_h\|_{L^2}$ & 00.018709 & 0.00089377 & 4.39 & 5.11742e-05 & 4.13 & 3.12506e-06 & 4.03  \\
%\cline{3-10}
% & & $|u-u_h|_{H^1}$   & 0.0668036 & 0.00720006 & 3.21 & 0.000860617 & 3.06 & 0.000106343 & 3.02  \\
\cline{3-10}
 & & $\|u-u_h\|_{L^\infty}$   & 0.00407847 & 0.000364257 & 3.48 & 2.42812e-05 & 3.91 & 1.54065e-06 & 3.98  \\
\hline
\end{tabular}\label{tab2dtk1uxx2}
\end{table}

\begin{table}[!htbp]\tabcolsep0.03in
\caption{2D $L^2, \; L^\infty$ errors for linearized Cahn-Hillard equation at $T = 0.1$, $a=\sqrt{3}/2$.}
\begin{tabular}[c]{||c|c|c|c|c|c|c|c|c|c||}
\hline
\multirow{2}{*}{$k$} & \multirow{2}{*}{$\Delta t$}&   \multirow{2}{*}{ } & N=8 & \multicolumn{2}{|c|}{N=16} & \multicolumn{2}{|c|}{N=32} & \multicolumn{2}{|c||}{N=64}  \\
\cline{4-10}
& & & error & error & order & error & order & error & order\\
%\hline
%\multirow{3}{*}{0}  & \multirow{3}{*}{1e-3} & $\|u-u_h\|_{L^2}$ &  1.08352 & 0.595592 & 0.86 & 0.375278 & 0.67 & 0.294579 & 0.35  \\
%\cline{3-10}
% & & $|u-u_h|_{H^1}$  & 4.12186 & 4.12186 & 0.00 & 4.12186 & 0.00 & 4.12186 & 0.00  \\
%\cline{3-10}
% & & $\|u-u_h\|_{L^\infty}$  & 0.346698 & 0.218631 & 0.67 & 0.134091 & 0.71 & 0.0949554 & 0.50  \\
%\hline
\hline
\multirow{2}{*}{1}  & \multirow{2}{*}{1e-3} & $\|u-u_h\|_{L^2}$ &  0.215662 & 0.0365488 & 2.56 & 0.00797165 & 2.20 & 0.0018959 & 2.07  \\
%\cline{3-10}
% & & $|u-u_h|_{H^1}$  & 1.1107 & 0.488877 & 1.18 & 0.236301 & 1.05 & 0.116837 & 1.02  \\
\cline{3-10}
 & & $\|u-u_h\|_{L^\infty}$  & 0.100838 & 0.0217418 & 2.21 & 0.00517092 & 2.07 & 0.00126682 & 2.03  \\
\hline
\hline
\multirow{2}{*}{2}  & \multirow{2}{*}{1e-4} & $\|u-u_h\|_{L^2}$ & 0.0476107 & 0.00759121 & 2.65 & 0.00102002 & 2.90 & 0.000129942 & 2.97  \\
%\cline{3-10}
% & & $|u-u_h|_{H^1}$  & 0.1914 & 0.0548671 & 1.80 & 0.0142959 & 1.94 & 0.00361414 & 1.98  \\
\cline{3-10}
 & & $\|u-u_h\|_{L^\infty}$  & 0.0147802 & 0.00225339 & 2.71 & 0.000294436 & 2.94 & 3.70339e-05 & 2.99  \\
 \hline
\hline
\multirow{2}{*}{$k$} & \multirow{2}{*}{$\Delta t$}&   \multirow{2}{*}{ } & N=4 & \multicolumn{2}{|c|}{N=8} & \multicolumn{2}{|c|}{N=16} & \multicolumn{2}{|c||}{N=32}  \\
\cline{4-10}
& & & error & error & order & error & order & error & order\\
\hline
\multirow{2}{*}{3} & \multirow{2}{*}{1e-5}  & $\|u-u_h\|_{L^2}$ & 0.0144092 & 0.000677035 & 4.41 & 3.87644e-05 & 4.13 & 2.36723e-06 & 4.03  \\
%\cline{3-10}
% & & $|u-u_h|_{H^1}$   & 0.0625901 & 0.00668012 & 3.23 & 0.000798434 & 3.06 & 9.86589e-05 & 3.02  \\
\cline{3-10}
 & & $\|u-u_h\|_{L^\infty}$   & 0.00388857 & 0.000338347 & 3.52 & 2.25334e-05 & 3.91 & 1.42943e-06 & 3.98  \\
\hline
\end{tabular}\label{tab2dtk1uxx3}
\end{table}

\end{example}

\begin{example}\label{Ex1dbd1} (Dirichlet boundary condition of the first  kind) We consider the following initial-boundary value problem
\begin{equation*}%\label{Ex1Biharmonicbd1}
\left \{
\begin{array}{rl}
u_t & =-u_{xxxx},\; (x, t)\in [0, 2\pi]\times (0, T],\\
u(x, 0) &=\sin x, \\
    u(0, t) & =  u(2\pi,t)=0,\\
    u_x(0,t) & = u_x(2\pi,t)=e^{-t},
\end{array}
\right.
\end{equation*}
which admits the exact solution $u(x,t) = e^{-t} \sin(x)$.
%\bq
%    u(x,t) = e^{-t} \sin(x).
%\eq

We test this  example using DG scheme (\ref{FPDGCell}) with boundary fluxes (\ref{fluxBD1}). We pay special attention on the effects of the boundary flux parameter $\beta_1$. The comparison results in Figure \ref{1dbeta1}
%Table \ref{tab1dtk1bd1} and \ref{tab1dtk1bd12}
show that the DG scheme with $\beta_1>0$ is optimally convergent, yet the scheme with $\beta_1=0$ only gives suboptimal orders of convergence for polynomials of degree $k$ with $k \geq 2$.
%; while the choices of $\beta_1$ are given in Table \ref{beta1dbd1}.
This test suggests that $\beta_1$ is necessary for $k\geq 2$  to weakly enforce the Dirichlet boundary data as formulated in (\ref{fluxBD1}), and $\beta_1=0$ is admissible for $k=1$.
Here, the convergence orders shown in Figure \ref{1dbeta1} are obtained based on total cell numbers $N=40, \ 80$ for $k \leq 2$ and $N=20, \ 40$ for $k \geq 3$.
%We can see that $\beta_1=0$ is admissible.
%the boundary fluxes (\ref{fluxBD}).
%This implies that $\beta_1$ is necessary in the boundary fluxes (\ref{fluxBD}).
%n this test, we want to show the necessity of the boundary fluxes parameter $\beta_1$ by comparing the errors and orders of accuracy at $t=1$ for $\beta_1=0$ and $\beta_1>0$. Here $\beta_1>0$ are chosen following Table \ref{beta1dbd1}.

 \begin{figure}
 \centering
 \subfigure{\includegraphics[width=0.49\textwidth]{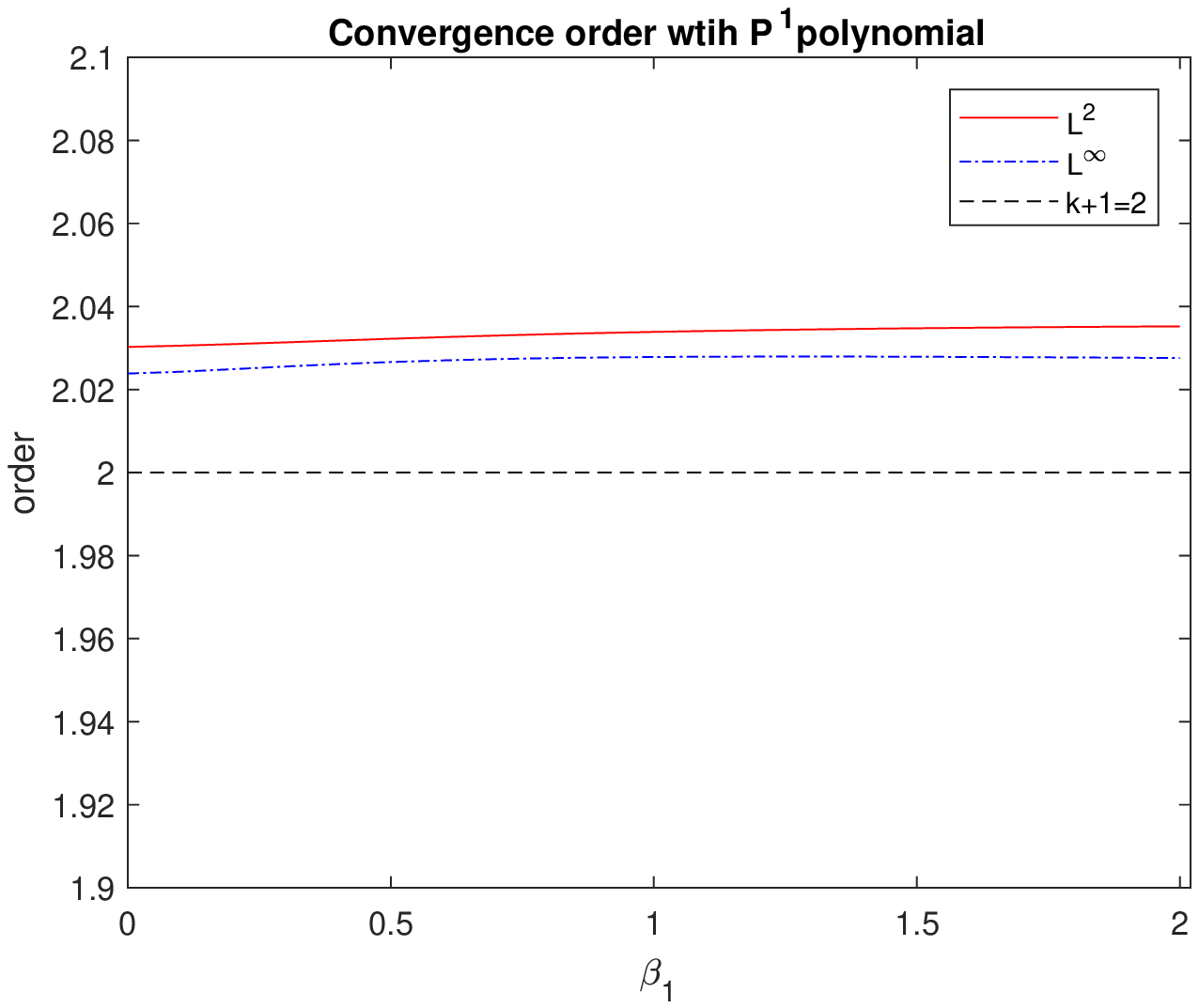}}
 \subfigure{\includegraphics[width=0.49\textwidth]{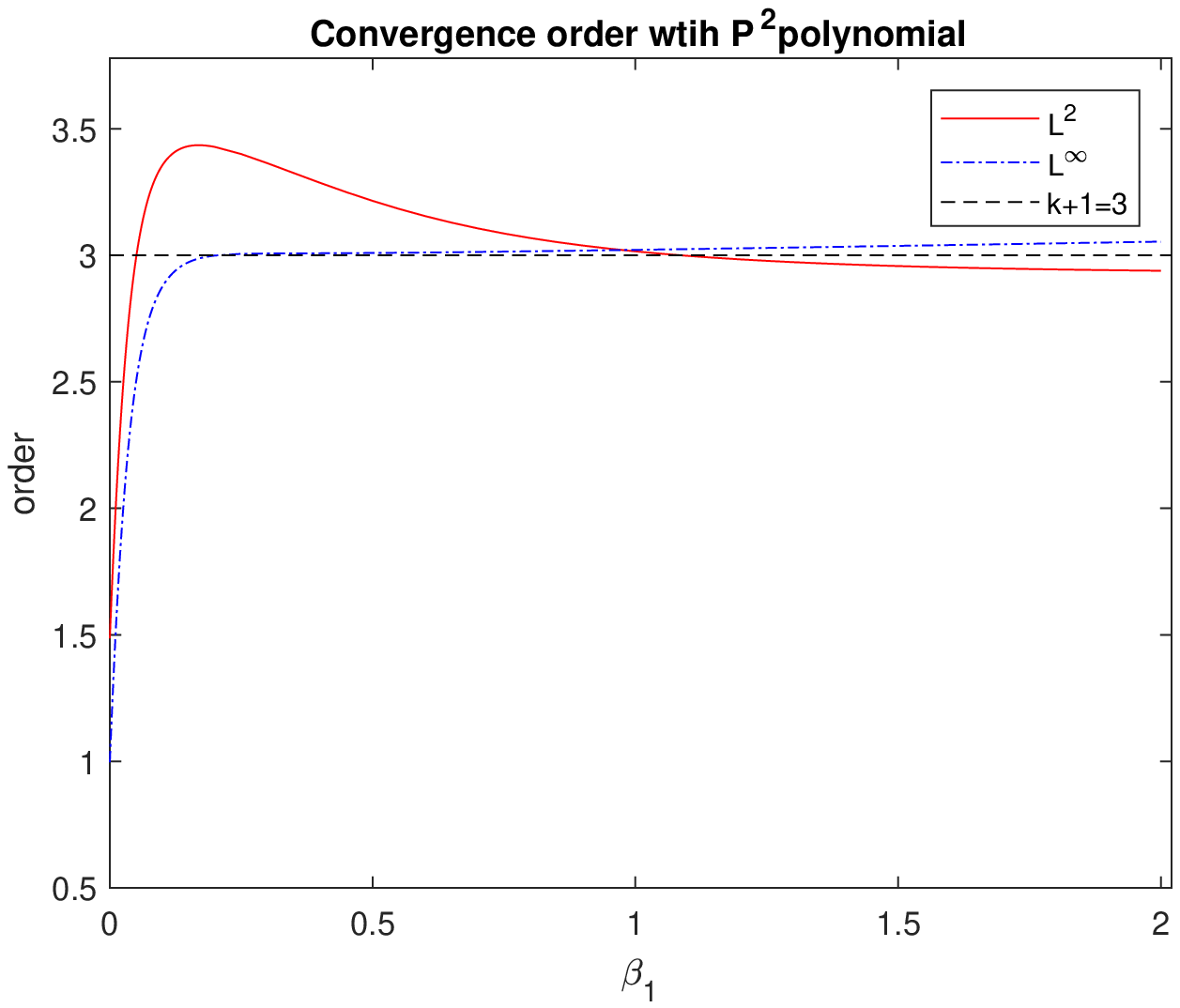}}
 \subfigure{\includegraphics[width=0.49\textwidth]{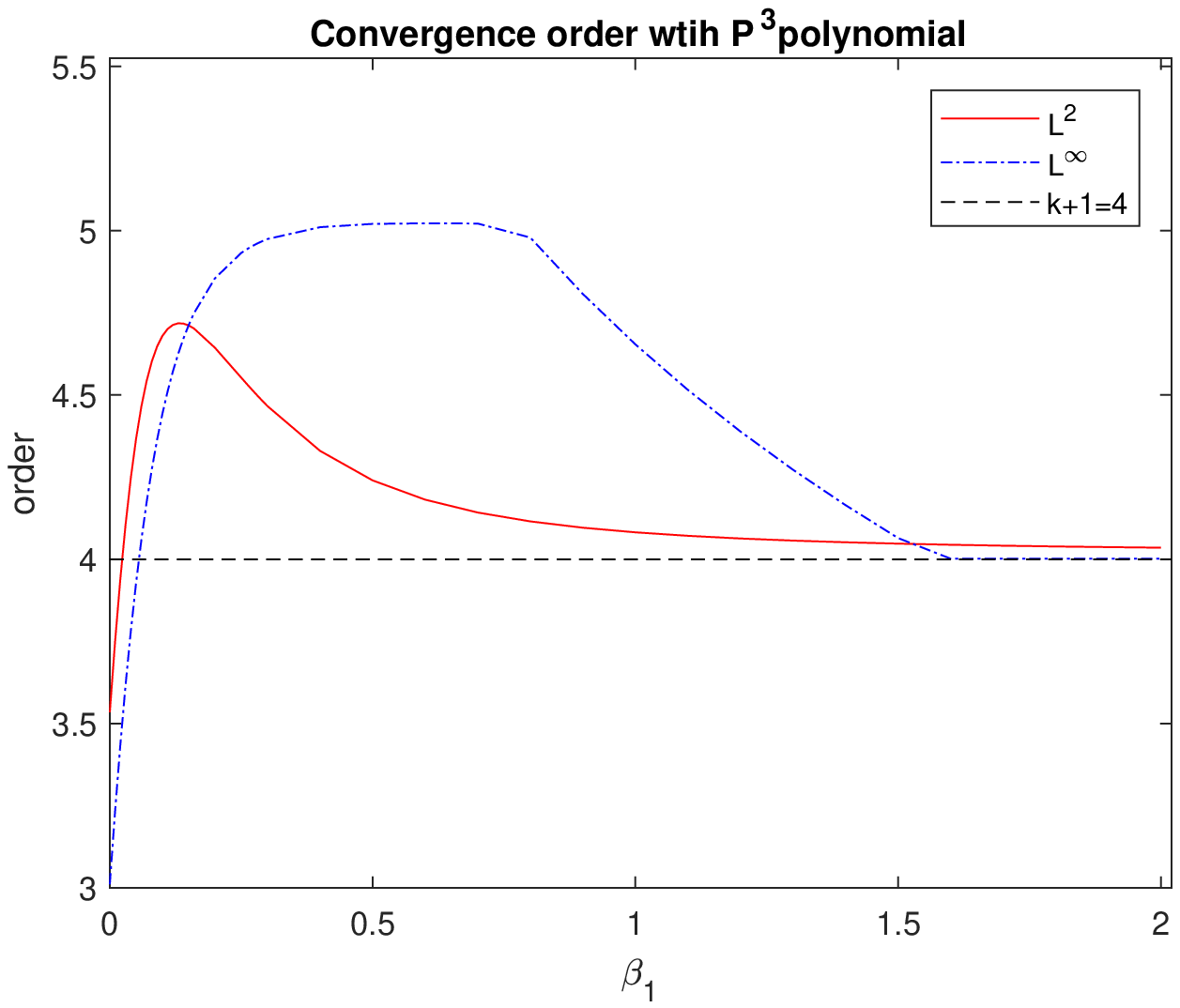}}
 \subfigure{\includegraphics[width=0.49\textwidth]{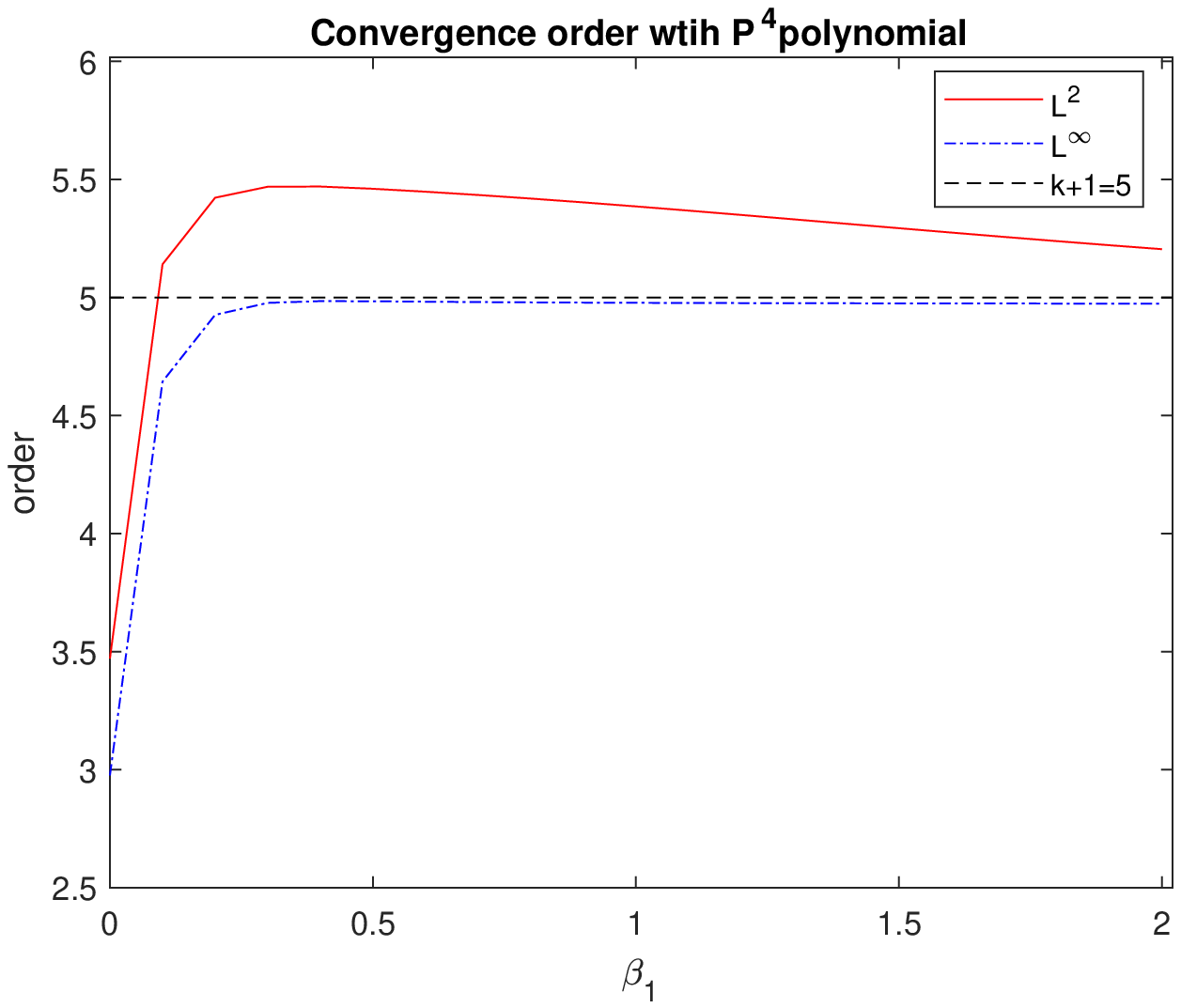}}
 \caption{The convergence orders with $P^k$ polynomials at $T=0.1$, Example 5.4.}% in 1D with BC of $1$st kind.}
 \label{1dbeta1}
 \end{figure}

\end{example}

\begin{example}\label{Ex1dbd2} (Dirichlet boundary condition of the second kind) We consider the following initial-boundary value problem
\begin{equation*}\label{Ex1Biharmonicbd2}
\left \{
\begin{array}{rl}
u_t & =-u_{xxxx},\; (x, t)\in [0, 3\pi]\times (0, T],\\
u(x, 0) &=\sin x, \\
    u(0, t) & =  u(3\pi,t)=0,\\
    u_{xx}(0,t) & = u_{xx}(3\pi,t)=0.
\end{array}
\right.
\end{equation*}
%The exact solution is given by
%\bq
%    u(x,t) = e^{-t} \sin(x).
%\eq
We test this example using DG scheme (\ref{FPDGCell}) with boundary fluxes (\ref{fluxBD2}),  with emphasis on the effects of the boundary flux parameters $\beta_0$. The numerical results are reported in Table \ref{tab1dtk1bd2}-\ref{tab1dtk1bd22} and Figure \ref{1dbeta0}.
%In Table \ref{tab1dtk1bd20} we see that the DG scheme with $\beta_0=0$ and  polynomials of degree one gives only suboptimal convergence no matter how we vary another parameter $\beta_1$.
In Table \ref{tab1dtk1bd2} we test the DG scheme based on $P^1$ polynomials, and we observe that the DG scheme with $\beta_0=0$ only gives suboptimal order of accuracy, while the DG scheme with other values of $\beta_0$ give optimal order of convergence in both $L^2$ and $L^\infty$ norms.
The comparison results in Table \ref{tab1dtk1bd2} show that $\beta_0$ is necessary for $k=1$ to  weakly enforce the Dirichlet boundary data as formulated in (\ref{fluxBD2}). Convergence orders in Figure \ref{1dbeta0},  obtained based on $P^1$ polynomials and total cell numbers $N=40, \ 80$, indicate that $|\beta_0| \geq C$ for some constants $C$ (e.g. $C=3$) is sufficient for the DG scheme to be optimally convergent.
%More exploration about the influence of $\beta_0$ to the convergence of the DG scheme is shown in Figure \ref{1dbeta0}.  These suggest the necessity of $\beta_0$ in the boundary fluxes (\ref{fluxBD2}).
However, extensive numerical tests indicate that $\beta_0=0$ is sufficient for the DG scheme with $k\geq 2$ to be optimally convergent,
%based on polynomials of degree $k$ with $k \geq 2$ to obtain $(k+1)$th order of accuracy,
see Table \ref{tab1dtk1bd22}. % and Figure \ref{1dbeta0}.
%Table \ref{tab1dtk1bd2} and \ref{tab1dtk1bd22} imply that $\beta_1$ is not necessary in obtaining the optimal errors in the boundary fluxes (\ref{fluxBD2}) if we follow the choice in Table \ref{tabbeta}, even though $\beta_1$ can improve the convergence order for some special choice of $\beta_0 >0$, see Figure \ref{1db0b1}.

 \begin{figure}
 \centering
 \subfigure{\includegraphics[width=0.49\textwidth]{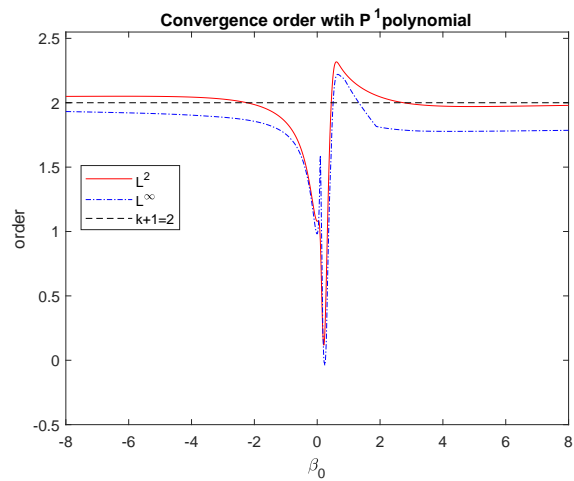}}
 %\subfigure{\includegraphics[width=0.49\textwidth]{figure/beta0p2.eps}}
 %\subfigure{\includegraphics[width=0.49\textwidth]{figure/beta0p3.eps}}
 %\subfigure{\includegraphics[width=0.49\textwidth]{figure/beta0p4.eps}}
 \caption{The convergence order with $P^1$ polynomials at $T=0.1$, Example 5.5.}% in 1D with BC of $2$nd kind.}
 \label{1dbeta0}
 \end{figure}

\begin{table}[!htbp]\tabcolsep0.03in
\caption{1D $L^2, \; L^\infty$ errors at $T= 1$ based on $P^1$ polynomials.}
\begin{tabular}[c]{||c|c|c|c|c|c|c|c|c|c||}
\hline
\multirow{2}{*}{$\beta_0$} & \multirow{2}{*}{$\Delta t$}&   \multirow{2}{*}{ } & N=10 & \multicolumn{2}{|c|}{N=20} & \multicolumn{2}{|c|}{N=40} & \multicolumn{2}{|c||}{N=80}  \\
\cline{4-10}
& & & error & error & order & error & order & error & order\\
\hline
\multirow{2}{*}{0.0}  & \multirow{2}{*}{1e-3} & $\|u-u_h\|_{L^2}$ &  0.0657588 & 0.0254149 & 1.37 & 0.0117346 & 1.11 & 0.00574456 & 1.03  \\
\cline{3-10}
%& & $|u-u_h|_{H^1}$  & 0.166228 & 0.0808126 & 1.04 & 0.040189 & 1.01 & 0.0200692 & 1.00  \\
%\cline{3-10}
& & $\|u-u_h\|_{L^\infty}$  & 0.0599114 & 0.0319409 & 0.91 & 0.0162757 & 0.97 & 0.00818006 & 0.99  \\
\hline
\hline
\multirow{2}{*}{0.4} & \multirow{2}{*}{1e-3}  & $\|u-u_h\|_{L^2}$ & 0.0766309 & 0.0587613 & 0.38 & 0.020537 & 1.52 & 0.0039176 & 2.39  \\
%\cline{3-10}
% & & $|u-u_h|_{H^1}$   & 0.162407 & 0.0762102 & 1.09 & 0.0384337 & 0.99 & 0.0177134 & 1.12  \\
\cline{3-10}
 & & $\|u-u_h\|_{L^\infty}$   & 0.0693423 & 0.0522882 & 0.41 & 0.0252481 & 1.05 & 0.00507969 & 2.31  \\
\hline
\hline
\multirow{2}{*}{1.0} & \multirow{2}{*}{1e-3} &  $\|u-u_h\|_{L^2}$ & 0.125475 & 0.0291785 & 2.10 & 0.00598832 & 2.28 & 0.00136537 & 2.13  \\
%\cline{3-10}
% & & $|u-u_h|_{H^1}$   &  0.168671 & 0.0745983 & 1.18 & 0.0352448 & 1.08 & 0.0174038 & 1.02  \\
\cline{3-10}
 & & $\|u-u_h\|_{L^\infty}$   &  0.133107 & 0.0370655 & 1.84 & 0.00778254 & 2.25 & 0.00196423 & 1.99  \\
\hline
\hline
\multirow{2}{*}{4.0} & \multirow{2}{*}{1e-3} &  $\|u-u_h\|_{L^2}$ & 0.0319942 & 0.00763239 & 2.07 & 0.00192231 & 1.99 & 0.000485668 & 1.98  \\
%\cline{3-10}
% & & $|u-u_h|_{H^1}$   &  0.148422 & 0.0706081 & 1.07 & 0.0348339 & 1.02 & 0.0173577 & 1.00  \\
\cline{3-10}
 & & $\|u-u_h\|_{L^\infty}$   &  0.0312461 & 0.00972619 & 1.68 & 0.00284823 & 1.77 & 0.000763535 & 1.90  \\
\hline
\hline
\multirow{2}{*}{-1.0}  & \multirow{2}{*}{1e-3} & $\|u-u_h\|_{L^2}$ &  0.0556805 & 0.0154405 & 1.85 & 0.00431204 & 1.84 & 0.00115667 & 1.90  \\
%\cline{3-10}
% & & $|u-u_h|_{H^1}$  & 0.155157 & 0.0720928 & 1.11 & 0.0350834 & 1.04 & 0.0173932 & 1.01  \\
\cline{3-10}
 & & $\|u-u_h\|_{L^\infty}$  & 0.0605392 & 0.0215281 & 1.49 & 0.00658334 & 1.71 & 0.00182776 & 1.85  \\
\hline
\end{tabular}\label{tab1dtk1bd2}
\end{table}

\begin{table}[!htbp]\tabcolsep0.03in
\caption{1D $L^2, \; L^\infty$ errors at $T = 1$ with $k \geq 2$ and $\beta_0=0$.}
\begin{tabular}[c]{||c|c|c|c|c|c|c|c|c|c||}
\hline
\multirow{2}{*}{$k$} & \multirow{2}{*}{$\Delta t$}&   \multirow{2}{*}{ } & N=10 & \multicolumn{2}{|c|}{N=20} & \multicolumn{2}{|c|}{N=40} & \multicolumn{2}{|c||}{N=80}  \\
\cline{4-10}
& & & error & error & order & error & order & error & order\\
\hline
\multirow{2}{*}{2}  & \multirow{2}{*}{1e-3} & $\|u-u_h\|_{L^2}$ & 0.00552409 & 0.00074947 & 2.88 & 9.59851e-05 & 2.96 & 1.20144e-05 & 3.00  \\
%\cline{3-10}
% & & $|u-u_h|_{H^1}$  & 0.0191142 & 0.00505481 & 1.92 & 0.00128452 & 1.98 & 0.000320981 & 2.00  \\
\cline{3-10}
 & & $\|u-u_h\|_{L^\infty}$  & 0.00354416 & 0.000492849 & 2.85 & 6.34389e-05 & 2.96 & 7.93492e-06 & 3.00  \\
\hline
\hline
\multirow{2}{*}{$k$} & \multirow{2}{*}{$\Delta t$}&   \multirow{2}{*}{ } & N=5 & \multicolumn{2}{|c|}{N=10} & \multicolumn{2}{|c|}{N=20} & \multicolumn{2}{|c||}{N=40}  \\
\cline{4-10}
& & & error & error & order & error & order & error & order\\
\hline
\multirow{2}{*}{3} & \multirow{2}{*}{1e-4}  & $\|u-u_h\|_{L^2}$ & 0.000793078 & 4.06207e-05 & 4.29 & 2.29901e-06 & 4.14 & 1.36679e-07 & 4.07  \\
%\cline{3-10}
% & & $|u-u_h|_{H^1}$   & 0.00328862 & 0.000369281 & 3.15 & 4.39214e-05 & 3.07 & 5.36395e-06 & 3.03  \\
\cline{3-10}
 & & $\|u-u_h\|_{L^\infty}$   & 0.000711528 & 4.52923e-05 & 3.97 & 2.84052e-06 & 4.00 & 1.75755e-07 & 4.01  \\
\hline
\hline
\multirow{2}{*}{4} & \multirow{2}{*}{1e-4} &  $\|u-u_h\|_{L^2}$ & 4.44683e-05 & 1.5007e-06 & 4.89 & 4.8107e-08 & 4.96 & 1.51427e-09 & 4.99  \\
%\cline{3-10}
% & & $|u-u_h|_{H^1}$   &   0.000233946 & 1.5414e-05 & 3.92 & 9.79734e-07 & 3.98 & 6.15172e-08 & 3.99  \\
\cline{3-10}
 & & $\|u-u_h\|_{L^\infty}$   &  2.76165e-05 & 1.01708e-06 & 4.76 & 3.33058e-08 & 4.93 & 1.05342e-09 & 4.98  \\
\hline
\end{tabular}\label{tab1dtk1bd22}
\end{table}

% \begin{figure}
% \centering
%% \subfigure{\includegraphics[width=0.49\textwidth]{figure/beta0beta1p1.eps}}
% \subfigure{\includegraphics[width=0.49\textwidth]{figure/beta0beta1p2.eps}}
% \subfigure{\includegraphics[width=0.49\textwidth]{figure/beta0beta1p3.eps}}
% \subfigure{\includegraphics[width=0.49\textwidth]{figure/beta0beta1p4.eps}}
% \caption{The convergence order with $P^k$ polynomials at $t=0.1$ in 1D with BC of $2$nd kind.}
% \label{1db0b1}
% \end{figure}

\end{example}

%\begin{example}\label{Ex2dbd1} (Dirichlet boundary condition of the first kind in 2D)
%\begin{equation}\label{Ex1Bi2Dbd1}
%\left \{
%\begin{array}{rl}
%    u_t + \Delta^2 u = & 0 \quad x \in [0, 4\pi]\times [0, 4\pi]\times (0, T],\\
%    u(x,y,0) = &  \sin(0.5x)\sin(0.5y),\\
%\end{array}
%\right.
%\end{equation}
% \begin{figure}
% \centering
% \subfigure{\includegraphics[width=0.49\textwidth]{figure/2dbeta1p1.eps}}
% \subfigure{\includegraphics[width=0.49\textwidth]{figure/2dbeta1p2.eps}}
% \subfigure{\includegraphics[width=0.49\textwidth]{figure/2dbeta1p3.eps}}
% \caption{The convergence order with $Q^k$ polynomials at $T=0.01$ in 2D with BC of $1$st kind.}
% \label{2dbeta1}
% \end{figure}
%
%\end{example}
%
%
%\begin{example}\label{Ex2dbd2} (Dirichlet boundary condition of the second kind in 2D)
%
% \begin{figure}
% \centering
% \subfigure{\includegraphics[width=0.49\textwidth]{figure/2dbeta0p1.eps}}
% \subfigure{\includegraphics[width=0.49\textwidth]{figure/2dbeta0p2.eps}}
% \subfigure{\includegraphics[width=0.49\textwidth]{figure/2dbeta0p3.eps}}
% \caption{ The convergence order with $Q^1$ polynomials at $T=0.01$ in 2D with BC of $2$nd kind. }
% \label{2dbeta02nd}
% \end{figure}
%
%\end{example}

%\begin{equation}\label{Ex1Biharmonicnon}
%\left \{
%\begin{array}{rl}
%    u_t = & -(\partial_x^2+1)^2u -f(u)   \quad x \in [0, 8\pi],\\
%    u(x,0) = &  u_0(x),\\
%    u(0,t) =& u(8\pi,t) = 0,\\
%    u_{xx}(0,t) =& u_{xx}(8\pi,t) = 0,\\
%\end{array}
%\right.
%\end{equation}
%where
%\bq
%f(u) = - 0.5u + u^3, \\
%\eq
\begin{example}\label{Ex1dNonPtn} (Pattern selection) For one-dimensional
Swift-Hohenberg equation, we consider the following problem
\begin{equation*}%\label{SHex}
\left \{
\begin{array}{rl}
% & u_t=-(\partial_x^2+1)^2u+f(u) \quad x\in [0, 8\pi], \\
 & u_t=-u -2u_{xx} -u_{xxxx}+f(u) \quad (x,t) \in [0, L] \times (0,T], \\
& u=0 \; \text{and} \; u_{xx}=0 \quad \text{at} \; x=0, L, \; t>0,\\
% & u(0,t)=  u(L,t)=0, \\
 %& u_{xx}(0,t)= u_{xx}(L,t)=0,\\
 & u(x,0)=0.1\sin\left(\frac{\pi x}{L} \right),
\end{array}
\right.
\end{equation*}
where $f(u) = \epsilon u - u^3$.  The asymptotic solution behavior of this problem was studied in \cite{PR04} with particular focus on the role of the parameter $\epsilon$ and the length $L$ of the domain on the selection of the limiting profile.
%$$
%f(u)=-\frac{d\Phi(u)}{du} = 0.5 u-u^3, \quad \Phi(u) = -\frac{1}{4}u^2+\frac{1}{4}u^4.
%$$
%In this test case, we take $[a,b]=[0,8\pi]$, $D=\kappa=1$, $\varepsilon=0.5, g=0$ and the initial $u_0(x)=2\sin(x/8)$.
We test the case of $\epsilon=0.5$ with $L=4, 14$, respectively,  and compare the results with those obtained in   \cite{PR04}.  This problem is solved by  DG scheme (\ref{FPDGFull1DNon}) based on polynomial $P^2$ with  $\delta=10^{-12}$. %We test this example until final time $T$ with mesh partition number $N$ and polynomial $P^2$,
The numerical solutions  shown in Figure \ref{1du12sol} display the pattern dynamics, which is consistent with the analysis and numerical tests in \cite{PR04}. The corresponding free energy dissipation  is shown in Figure \ref{1du12eng}.

 \begin{figure}
 \centering
 \subfigure[]{\includegraphics[width=0.49\textwidth]{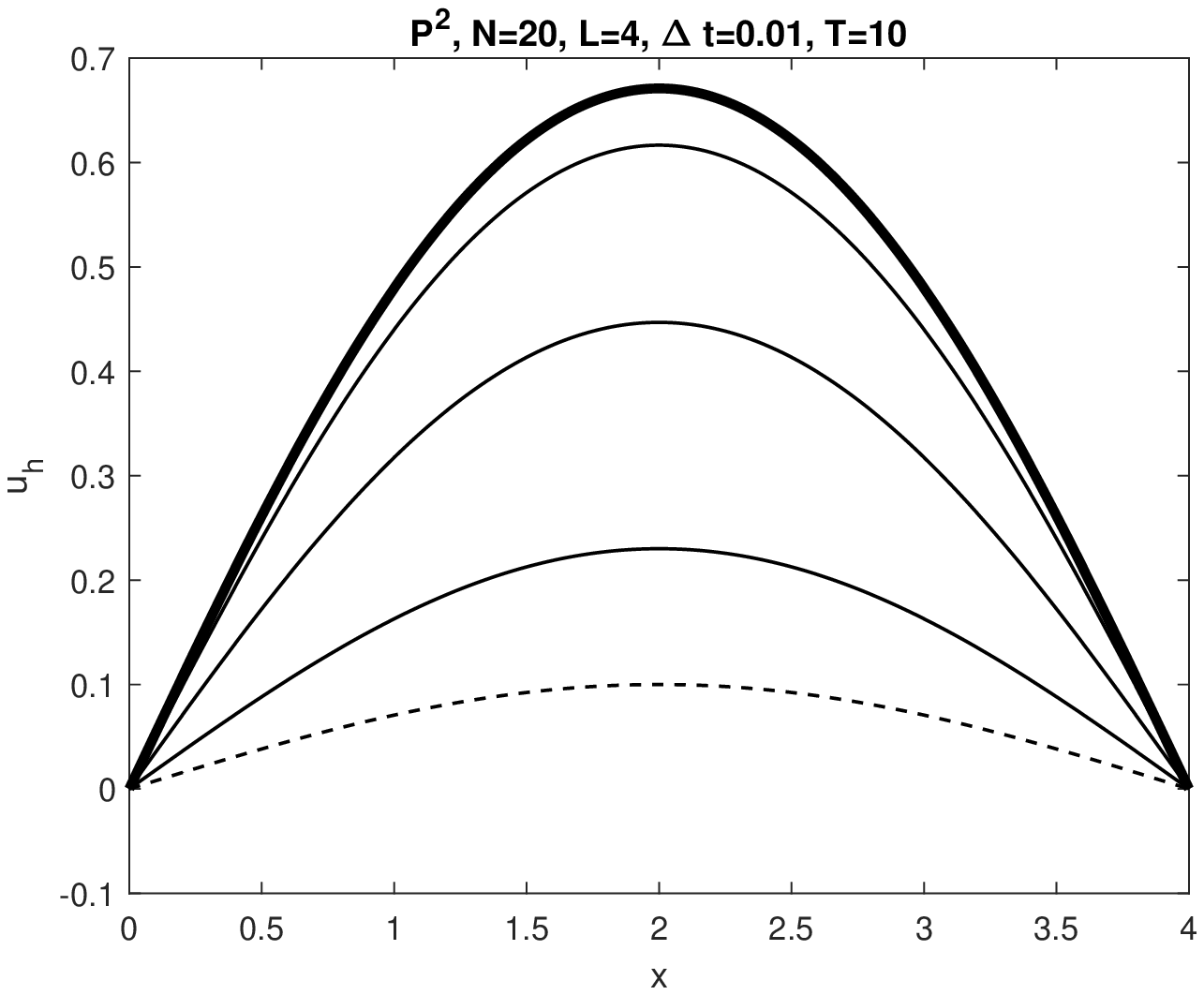}}
 \subfigure[]{\includegraphics[width=0.49\textwidth]{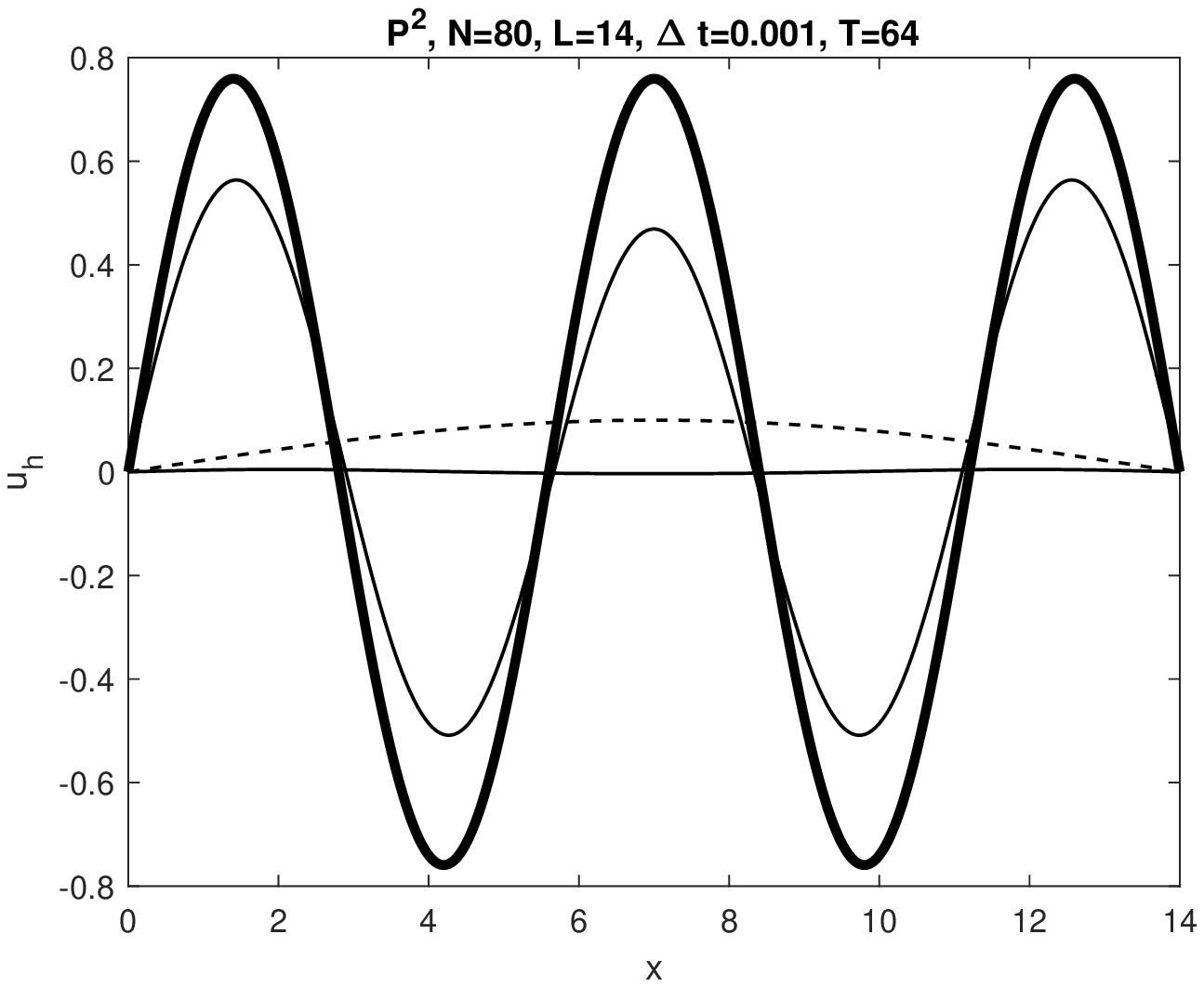}}
  \caption{ Evolution of patterns with (a) $L=4.0$, (b) $L=14.0$. The dashed curve is the initial pattern and the thick curve the final pattern. The other curves represent patterns at the intermediate times.
  } \label{1du12sol}
 \end{figure}

 \begin{figure}
 \centering
 \subfigure[]{\includegraphics[width=0.49\textwidth]{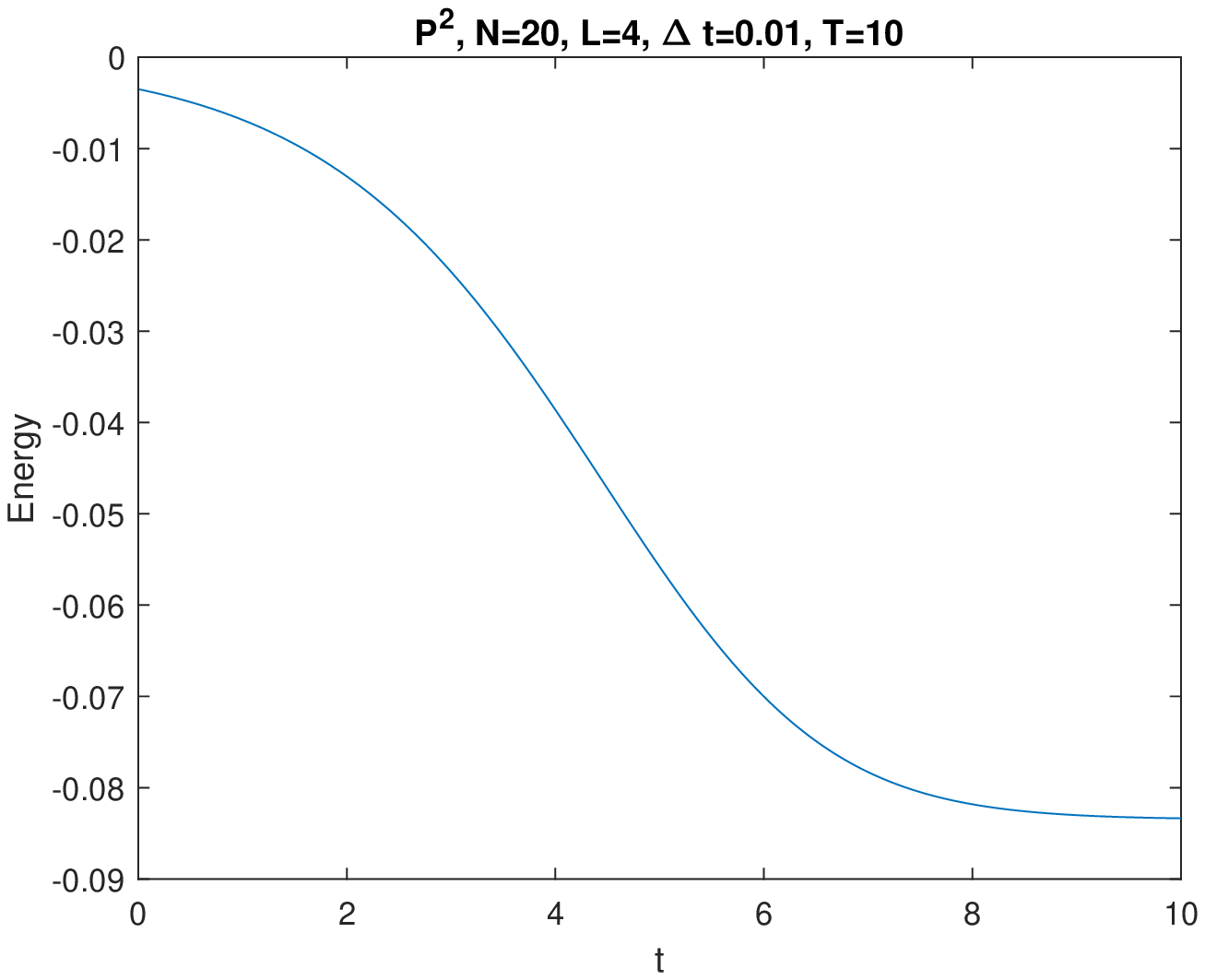}}
 \subfigure[]{\includegraphics[width=0.49\textwidth]{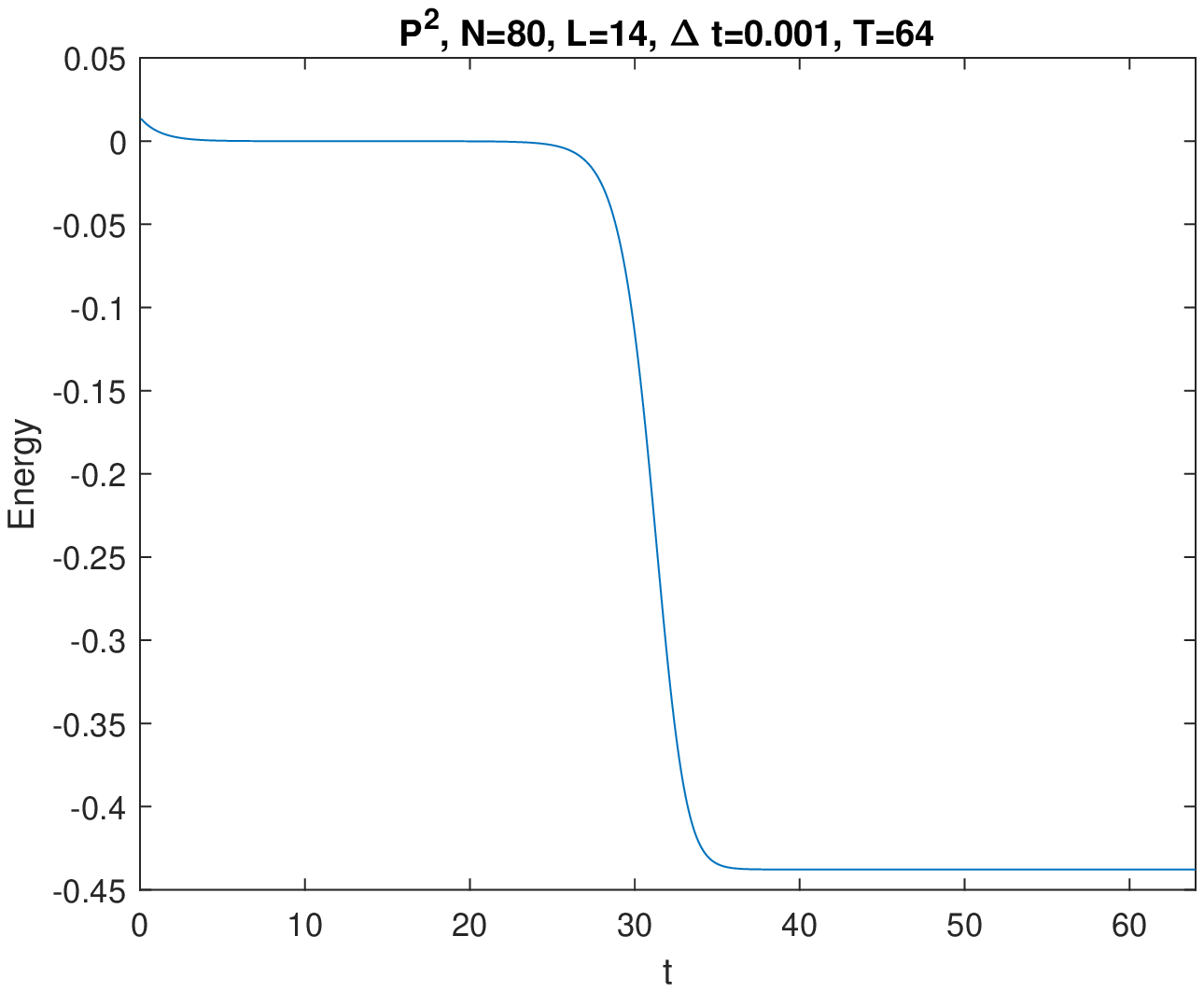}}
 \caption{ Energy evolution with (a) L=4.0, (b) $L=14.0$.
 } \label{1du12eng}
 \end{figure}

%\noindent\textbf{Test case 2.} We take the initial $U_0(x)=|\sin(8x)|$ in this case, and we tested this example until $t=20.0$ with polynomial $P^3$ and $N=20$,
%the PFDG solutions at $t=0, \ 5.0, 10.0, \ 15.0, \ 20.0$ are shown in Figure \ref{1du12sol}, and the corresponding free energy is shown in Figure \ref{1du12eng}. From Figure \ref{1du12sol}, we can find that the steady state pattern has formed at $t=15$.

%\begin{rem}
%In this example, we only considered one-dimensional Swift-Hohenberg equation, we will explore more nonlinear 4th order PDEs in our following work.
%\end{rem}

\end{example}

\section{Concluding remarks}
A novel discontinuous Galerkin (DG) method without interior penalty has been proposed to solve the time-dependent fourth order partial differential equations. For the biharmonic equation, the DG scheme is based on the mixed formulation of the original model.
%By rewriting the equation into a second order system, it spatially discretes the weak formulation of the system by DG scheme without penalty terms and takes a list of time discrezation for time matching.
Both stability and optimal $L^2-$error estimates of the DG method are proved in both one-dimensional and multi-dimensional settings subject to periodic boundary conditions. Extensions to general fourth order equations and cases with three typical non-homogeneous boundary conditions are discussed, following by an application to solving the one-dimensional Swift-Hohenberg equation, which admits a decay free energy.
Several numerical results are presented to verify the stability and accuracy of the schemes.

 \bigskip
\section*{Acknowledgments} This research was partially supported by the National Science Foundation under Grant DMS1312636.

\end{document}